\documentclass[12pt]{amsart}
\usepackage{amssymb,amsmath,amsthm}
\numberwithin{equation}{section}
\begin{document}

\theoremstyle{plain}
\newtheorem{theorem}{Theorem}[section]
\newtheorem{lemma}[theorem]{Lemma}
\newtheorem{proposition}[theorem]{Proposition}
\newtheorem{corollary}[theorem]{Corollary}
\newtheorem{conjecture}[theorem]{Conjecture}

\theoremstyle{definition}
\newtheorem*{definition}{Definition}

\theoremstyle{remark}
\newtheorem*{remark}{Remark}
\newtheorem{example}{Example}[section]
\newtheorem*{remarks}{Remarks}

\newcommand{\cc}{{\mathbf C}}
\newcommand{\qq}{{\mathbf Q}}
\newcommand{\rr}{{\mathbf R}}
\newcommand{\nn}{{\mathbf N}}
\newcommand{\zz}{{\mathbf Z}}
\newcommand{\pp}{{\mathbf P}}
\newcommand{\al}{\alpha}
\newcommand{\be}{\beta}
\newcommand{\ga}{\gamma}
\newcommand{\ze}{\zeta}
\newcommand{\om}{\omega}
\newcommand{\ep}{\epsilon}
\newcommand{\la}{\lambda}
\newcommand{\de}{\delta}
\newcommand{\De}{\Delta}
\newcommand{\Ga}{\Gamma}
\newcommand{\si}{\sigma}

\title{Blenders}

\author{Bruce Reznick}
\address{Department of Mathematics, University of 
Illinois at Urbana-Champaign, Urbana, IL 61801} 
\email{reznick@math.uiuc.edu}

 \dedicatory{Dedicated to the memory of Julius Borcea}

\begin{abstract}
A blender is a closed convex cone of real homogeneous polynomials that is
also closed under linear changes of variable. Non-trivial blenders only
occur in even degree. Examples include the cones of psd forms, sos
forms, convex forms and sums of $2u$-th powers of forms of degree $v$.
We present some general properties of blenders and analyze the
extremal elements of some specific blenders. 
\end{abstract}
\date{\today}
\maketitle

\maketitle

\section{Introduction and Overview}

Let $F_{n,d}$ denote the vector space of real homogeneous forms
$p(x_1,\dots,x_n)$ of degree $d$.
A blender is a closed convex cone in $F_{n,d}$
which is also closed under linear changes of variable. Blenders were
introduced in \cite{Re1} to help  describe several different familiar
cones of polynomials, but that memoir was mainly concerned with the
cones of psd and sos forms and their duals, and the  discussion of
blenders {\it per se} was scattered there  (pp.\! 36-50, 119-120,
140-142). This paper is devoted to a general discussion of blenders 
and their properties, as well as considering the extremal elements of
some particular blenders not discussed in \cite{Re1}.

Non-trivial blenders will only occur when $d = 2r$ is an even integer.
Choi and Lam \cite{CL1,CL2} named the cone of {\it psd} forms:
\begin{equation}
P_{n,2r}:= \{p \in F_{n,2r} :  u \in \rr^n \implies p(u) \ge 0\},
\end{equation}
and the cone of  {\it sos} forms:
\begin{equation}
\Sigma_{n,2r}:= \biggl\{p \in F_{n,2r} : p = \sum_{k=1}^s h_k^2,\  h_k
\in F_{n,r}\biggr\}.
\end{equation}
Other blenders of interest in \cite{Re1} are the cone of sums of $2r$-th powers:
\begin{equation}
Q_{n,2r}:= \biggl\{p \in F_{n,2r} : p = \sum_{k=1}^s (\al_{k1}x_1 +
\cdots + \al_{kn}x_n)^{2r},\ \al_{kj} \in \rr\biggr \}
\end{equation}
and the ``Waring blenders'': suppose $r = uv$, $u,v \in \nn$ and let:
\begin{equation}
W_{n,(u,2v)}:= \biggl\{p \in F_{n,2r} : p = \sum_{k=1}^s  h_k^{2v},\  h_k
\in F_{n,u}\biggr \}.
\end{equation}
Note that  $W_{n,(r,2)} = \Sigma_{n,2r}$ and $W_{n,(1,2r)} =
Q_{n,2r}$.

The Waring blenders generalize. If $d = 2r$ and $\sum_{i=1}^m u_iv_i =
r$, let
\begin{equation}
W_{n,\{(u_1,2v_1),\dots, (u_m,2v_m)\}}:= \biggl\{p \in F_{n,2r} : p =
\sum_{k=1}^s  h_{k,1}^{2v_1}\cdots h_{k,m}^{2v_m} ,\  h_{k,i}  
\in F_{n,u_i} \biggr\}.
\end{equation}
There has been recent interest in the cones of convex forms: 
\begin{equation}\label{E:kn2r}
K_{n,2r}:= \{p \in F_{n,2r} : p \ \text{is convex}\}.
\end{equation}
We shall use the two equivalent definitions of ``convex'' (see
e.g. \cite[Thm.4.1,4.5]{Ro}): under the 
{\it line segment} definition, $p$ is convex if for all $u, v
\in \rr^n$ and $\la \in [0,1]$,
\begin{equation}
p(\la u + (1 - \la) v) \le \la p(u) + (1-\la)p(v).
\end{equation}
The {\it Hessian} definition says that if 
\begin{equation}\label{E:hes}
Hes(p;u,v):= \sum_{i=1}^n \sum_{j=1}^n \frac{\partial^2p}{\partial
  x_i \partial x_j}(u) v_iv_j,
\end{equation}
then $p$ is convex provided $Hes(p;u,v) \ge 0$ for all $u, v \in
\rr^n$. The cone $K_{n,m}$ appeared in \cite{Re1}, but as
$N_{n,m}$ (see Corollary 4.5). Pablo Parrilo asked
whether every convex form is  sos; that is, is $K_{n,2r} \subseteq
\Sigma_{n,2r}$? This question 
has been answered by Greg Blekherman \cite{B} in the negative. For
fixed $n$, the ``probability'' that a convex form is sos goes to 0 
as $r \to \infty$.  No examples of $p \in
K_{n,2r} \setminus \Sigma_{n,2r}$ are yet known. 

 We now make the definition of blender more precise. 
Suppose $n \ge 1$ and $d \ge 0$. The index set for monomials in
$F_{n,d}$ consists of $n$-tuples of non-negative integers:  
\begin{equation}
 \mathcal I(n,d) = \biggl\lbrace i=(i_1,\dots,i_n): \sum\limits_{k=1}^n
 i_k = d\biggr\rbrace.
\end{equation}
Write $N(n,d) = \binom {n+d-1}{n-1} = |\mathcal I(n,d)|$ and for $i
\in \mathcal I(n,d)$, let 
$c(i) = \frac{d!}{i_1!\cdots i_n!}$ be the associated multinomial coefficient.
The  abbreviation $u^i$ means $u_1^{i_1}\dots u_n^{i_n}$,
where $u$ may be an $n$-tuple of constants or variables.  
Every $p \in F_{n,d}$ can be written as
 \begin{equation}
p(x_1,\dots,x_n)=\sum_{i\in\mathcal I(n,d)} c(i)a(p;i)x^i.
\end{equation}
The identification of $p$ with the $N(n,d)$-tuple $(a(p;i))$ shows that 
$F_{n,d} \approx \rr^{N(n,d)}$ as a vector space. The topology 
placed on $F_{n,d}$ is the usual one: $p_m \to p$ means that for
every $i \in \mathcal I(n,d)$, $a(p_m;i) \to a(p;i)$.

For $\al \in \rr^n$, define $(\al\cdot)^d \in F_{n,d}$ by
\begin{equation}
(\al\cdot)^d(x) = \biggl(\sum_{k=1}^n \al_kx_k\biggr)^d =
\sum_{i\in\mathcal I(n,d)} c(i)\al^ix^i. 
\end{equation}
If $\al$ is regarded as a row vector and $x$ as a column vector,
then $(\al \cdot)^d(x) = (\al x)^d$.
If  $M = [m_{ij}]\in Mat_n(\rr)$ is a (not 
necessarily invertible) real $n\times n$ matrix and $p \in F_{n,d}$, we
define $p\circ M \in F_{n,d}$ by
\begin{equation}
  (p\circ M)(x_1,\dots,x_n)= p(\ell_1,\dots,\ell_n), \qquad
  \ell_j(x_1,\dots,x_n) =   \sum_{k=1}^nm_{jk}x_k.
\end{equation}
If $x$ is viewed as a column vector, then $(p\circ M)(x) =
p(Mx)$;  $(\al\cdot)^d \circ M = (\al M \cdot)^d$.

 Define $[[p]]$ to be $\{p \circ M: M \in Mat_n(\rr)\}$, 
the {\it closed orbit of $p$}. If $ p = q \circ M$ for {\it invertible}
$M$, we write $p \sim q$; invertibility implies that $\sim$ is an
equivalence relation.

\begin{lemma}

\smallskip

\ 

\noindent (i) If $p \in F_{n,d}$ and $d$ is odd, then $p \sim \la p$ for every $0
\neq \la \in \rr$. 

\noindent (ii) If   $p \in F_{n,d}$ and $d$ is even, then $p \sim \la p$
for every $0 < \la \in \rr$. 

\noindent (iii) If $u, \al \in \rr^n$, then there exists a (singular)
$M$ so that $p\circ M = p(u)(\al\cdot)^d.$ 
\end{lemma}

\begin{proof}
For (i), (ii), observe that  $(p \circ (cI_n)) =
c^dp$ since $p$ is homogeneous, and $cI_n$ is invertible if $c
\neq 0$. For (iii), note that if   $m_{jk} = u_j\al_k$
for $1 \le j,k \le n$, then 
\begin{equation}
 \ell_j(x) = u_j\sum_{k=1}^n \al_k x_k = (\al x)u_j \implies  (p\circ
 M)(x_1,\dots,x_n) = (\al x)^dp(u_1,\dots, u_n)
\end{equation} 
by homogeneity.
\end{proof}
\begin{definition}
A set $B \subseteq F_{n,d}$ is a {\it blender} if these conditions hold:

\smallskip
\noindent (P1) If $p, q \in B$, then $p+q   \in B$. 

\noindent  (P2) If $p_m \in B$ and $p_m \to p$, then $p \in
B$. 

\noindent  (P3) If $p \in B$ and  $M \in Mat_n(\rr)$, then $p
  \circ M \in B$.
\end{definition}
Thus, a blender is a closed convex cone of forms which is also
a union of closed orbits. Lemma 1.1 makes it unnecessary
to specify in (P1) that $p \in B$  and $\la \ge 0$ imply $\la p \in
B$. Let $\mathcal B_{n,d}$ denote the set of blenders in $F_{n,d}$.
Trivially, $\{0\}, F_{n,d} \in \mathcal B_{n,d}$. 

It is simple to see that $P_{n,2r}$ is a blender: conditions (P1) and
(P2) can be verified pointwise and if $p(u) \ge
0$ for every $u$, then the same will be true for $p(Mu)$. 
Similarly, $K_{n,2r}$ is a blender because (P1) and (P2) follow from
the  Hessian definition and (P3) follows from the line segment definition.

If $B_1, B_2 \in \mathcal B_{n,d}$, then  $B_1 \cap B_2 \in \mathcal
B_{n,d}$. Define the {\it Minkowski sum} 
\begin{equation}\label{E:b1+b2}
B_1+B_2:= \{p_1+p_2: p_i \in B_i\}.
\end{equation}
 The smallest blender containing both $B_1$ and $B_2$ must
include $B_1+B_2$; this set is a blender (Theorem 3.5(i)), but it
requires an argument to   prove  (P2). It is not hard to see that  
$\mathcal B_{n,d}$ is not always a chain. Let $(n,d) = (2,8)$ and let $B_1
=W_{2,\{(1,6),(1,2)\}}$ and $B_2 =  W_{2,\{(1,4),(1,4)\}}$. Then $x^6y^2 \in
B_1$ and $x^4y^4 \in B_2$. If $x^6y^2 \in B_2$, then
\begin{equation}
x^6y^2 = \sum_{k=1}^s(\al_k x + \be_k y)^4(\ga_k x + \de_k y)^4.
\end{equation}
A consideration of the coefficients of $x^8$ and $y^8$ shows that
$\al_k\ga_k = \be_k\de_k=0$ for all $k$, hence the only non-zero
summands are positive multiples of $x^4y^4$. Thus $x^6y^2 \not\in
B_2$, and, similarly, $x^4y^4 \not\in B_1$, so $B_1 \setminus B_2$ and $B_2
\setminus B_1$ are both non-empty. It is not clear which octics belong to
$B_1 \cap B_2$ and $B_1 + B_2$.
If $B_1 \in \mathcal B_{n,d_1}$
  and $B_2 \in \mathcal B_{n,d_2}$,  define
\begin{equation}\label{E:b1*b2}
B_1*B_2:= \left\{\sum_{k=1}^s  p_{1,k}p_{2,k}: p_{i,k} \in B_i \right\}.
\end{equation}
Again, this is a blender (Theorem 3.5(ii)), but (P2) is not trivial to prove.

We review some standard facts about convex cones; see \cite[Ch.2,3]{Re1}
and \cite{Ro}. 
If $C \subset \rr^N$ is a closed convex cone, then $u \in C$ is {\it
  extremal} if $u = v_1 + v_2, v_i \in C$, implies that $v_i = \la_i
u$, $\la_i \ge 0$. The set of extremal elements in $C$ is denoted
$\mathcal E(C)$.
All cones $C \neq 0, \rr^N$ in this
paper have the property that $x, -x \in C$ implies $x
= 0$. In such a cone, every element in $C$ is a sum of
extremal elements.  (It will follow from Prop.\!\!  2.4 that if $B \in \mathcal
B_{n,d}$  and $p,-p \in B$ for some $p \neq 0$, then $B = F_{n,d}$.)

 As usual, $u$ is {\it interior} to $C$  if $C$ contains a
 non-empty open ball  centered at $u$. The set of interior points of
$C$ is denoted $int(C)$, and the boundary of $C$ is denoted
$\partial(C)$. The next definition depends on the inner product. If $C$
is a closed convex cone, let
\begin{equation}
C^* = \{ v \in \rr^N : [u,v] \ge 0\quad \text{for all}\quad u \in C\}.
\end{equation}
Then $C^* \subset \rr^N$ is also a closed convex cone and $(C^*)^* =
C$; $C$ and $C^*$ are {\it dual} cones.

If $u \in C$ (and $\pm x \in C$ implies $x = 0$),
then $u \in int(C)$ if and only if $[u,v]>0$ for every
$0 \neq v \in C^*$ (see e.g. \cite[p.26]{Re1}). Thus, if $u
\in \partial(C)$  (in particular, if $u$ is extremal), then there
exists $v \in C^*$, $v \neq 0$ so that $[u,v] = 0$. 

This discussion applies to blenders by identifying $p \in F_{n,d}$ with
the $N(n,d)$-tuple of its coefficients. For example, $p \in int(B)$ if
there exists $\epsilon >0$ so that if $|a(q;i)| < \epsilon$ for all $i
\in \mathcal I(n,d)$, then $p +
q \in B$.  If $p \sim q \in B$, then $p$ and $q$ simultaneously belong
to (or do not belong to) $int(B), \partial(B), \mathcal E(B)$.
 We shall discuss in section two the natural inner product
on $F_{n,d}$. It turns out that, under this inner product, $P_{n,2r}$
and $Q_{n,2r}$ are 
dual cones (Prop.\!\! 3.8), as are $K_{n,2r}$ and
$W_{n,\{(1,2r-2),(1,2)\}}$ (Theorem 3.11).

The description of $\mathcal E(P_{n,2r})$ is extremely difficult if $n
\ge 3$. (See e.g \cite{CL1, CL2, CLRsex,CLR, Ha, ReAGI,Re4}.) Every element of
$\mathcal E(\Sigma_{n,2r})$ obviously has the form $h^2$, but not
every square is extremal; e.g., 
\begin{equation}\label{E:h22}
(x^2+y^2)^2 = (x^2-y^2)^2 + (2xy)^2 =\frac1{18} \left((\sqrt 3\ x +
y)^4 + (\sqrt 3\ x - y)^4 + 16y^4 \right). 
\end{equation} 

We now describe the contents of this paper. Section two reviews the
relevant results from \cite{Re1} regarding
 the inner product and its many properties. The
principal results are that if $B \in \mathcal B_{n,d}$ and $B \neq \{0\},
F_{n,d}$, then $d=2r$ is even and $Q_{n,2r} \subset \pm B \subset
  P_{n,2r}$ (Prop.\!\!  2.5);  the dual cone to a blender is also a
blender (Prop.\!\!  2.7). Section three begins with a number of
preparatory lemmas, mainly involving convergence. We show that if
$B_i$ are blenders, then so are $B_1+B_2$ and $B_1*B_2$ (Theorem 3.5)
and hence the Waring blenders and their generalizations are blenders
(Theorems 3.6, 3.7). We show that $P_{n,2r}$ and $Q_{n,2r}$ are dual
and give a description of $W_{n,(u,v)}^*$ (both from \cite{Re1}) and
show that $K_{n,2r}$ and $W_{n,\{(1,2r-2),(1,2)\}}$ are dual (Theorem
3.11). In section four, we consider $K_{n,2r}$. We show that it cannot
be decomposed non-trivially as $B_1*B_2$ (Corollary 4.2), and 
that $K_{n,2r}=N_{n,2r}$  (c.f.\! \eqref{E:kn2r}, \eqref{E:nnd},
Corollary 4.5). We also show that if $p$ is positive definite, then $(\sum
x_i^2)^Np$ is convex for sufficiently large $N$ (Theorem
4.6). In section five, we show that (up to $\pm$) $\mathcal B_{2,4}$
consists of a one-parameter family of blenders $B_{\tau}$, $\tau \in
[-\frac 13, 0]$, where $\tau = \inf\{\la: x^4 + 6\la x^2y^2 + y^4 \in
B_{\tau}\}$, increasing from $Q_{2,4}=B_0$ to $P_{2,4}=B_{-\frac 13}$,
and that $B_{\tau}^* = B_{U(\tau)}$, where $U(\tau) =
  -\frac{1+3\tau}{3-3\tau}$ (Theorem 5.7). In 
section six, we review the results of $K_{2,4}$ and $K_{2,6}$ in
\cite{D1,D2,Re00} by Dmitriev and the author, and give some new
examples in $\partial(K_{2,2r})$.  The 
full analysis of $\mathcal E(K_{2,2r})$ seems intractable for $r \ge
4$. Finally, in section seven, we look at sums of 4th
powers of binary forms. Conjecture 7.1 states that $p \in W_{2,(u,4)}$
if and only if $p = f^2 + g^2$, where $f,g \in P_{n,2u}$. We show that
this is true for $u=1$ and for even symmetric octics $p$ (Theorems
7.3, 7.4). Our classification of even symmetric octics implies that
\begin{equation}\label{E:48}
x^8 + \al x^4y^4 + y^8 \in W_{2,(2,4)} \iff \al \ge - \tfrac {14}9.
\end{equation}

I would like to thank the organizers of  BIRS
10w5007, Convex Algebraic Geometry, held at Banff in February, 2010,
for the opportunity to speak. I would also like to thank  my fellow
participants for many stimulating conversations. Sections four and six were
particularly influenced by this meeting. I also thank Greg Blekherman for
very helpful email discussions. Special thanks to  Peter Kuchment, a 
classmate of V. I. Dmitriev, for trying contact him for
me. Finally, I thank the editors of this volume for the opportunity to
contribute to this memorial volume in memory of Prof. Borcea.

\section{The inner product}

For $p$ and $q$ in $F_{n,d}$, we define an inner product with deep
roots in 19th century algebraic geometry and analysis. Let
\begin{equation}\label{E:ip}
[p,q] = \sum_{i \in \mathcal I(n,d)} c(i)a(p;i)a(q;i). 
\end{equation}
This is the usual Euclidean inner product, if $p \leftrightarrow
(c(i)^{1/2}a(p;i)) \in \rr^N$. The many properties of this inner
product (see Props.\!\! 2.1, 2.6 and 2.9) strongly suggest that this
is the ``correct'' inner product for $F_{n,d}$. We present without
proof the following observations about the inner product.
\begin{proposition}\cite[pp.2,3]{Re1}
\ 

\noindent (i) $[p,q] = [q,p]$.

\noindent (ii) $j \in \mathcal I(n,d) \implies [p,x^j] = a[p;j]$.

\noindent (iii) $\al \in \rr^n \implies [p,(\al\cdot)^d] = p(\al)$.

\noindent (iv) If $p_m \to p$, then $[p_m,q] \to [p,q]$ for every $q \in
  F_{n,d}$.

\noindent (v) In particular, taking $q = (u \cdot)^d$, $p_m \to p \implies
  p_m(u) \to p(u)$ for all $u \in \rr^n$.
\end{proposition} 

The orthogonal complement of a subspace $U$ of $F_{n,d}$, 
\begin{equation}
U^\perp = \{ v \in F_{n,d}: [u,v] = 0\quad \text{for all}\quad u \in U\}, 
\end{equation}
 is  also a subspace of $F_{n,d}$ and $(U^\perp)^\perp = U$. 
The following result is widely-known and has been frequently proved over the
last century, see e.g.\cite[p.30]{Re1}. 
\begin{proposition}\cite[p.93]{Re1}
Suppose $S \subset \rr^n$ has non-empty interior. Then
$F_{n,d}$ is spanned by $\{(\al\cdot)^d: \al \in S \}$.
\end{proposition}
\begin{proof}
Let $U$ be the subspace of $F_{n,d}$ spanned by $\{(\al\cdot)^d:
\al \in S \}$ and suppose  
$q \in U^{\perp}$. Then  $0 = [q,(\al\cdot)^d] = q(\al)$ for all $\al \in S$.
Since $q$ is a polynomial which vanishes on an open set, $q = 0$. 
Thus, $U^{\perp} = \{0\}$, so $U = (U^\perp)^\perp = \{0\}^\perp = F_{n,d}$.
 \end{proof}

\begin{proposition}[Biermann's Theorem]\cite[p.31]{Re1}
The set $\{(i \cdot)^d : i \in \mathcal I(n,d)\}$ is a basis for
$F_{n,d}$.  
\end{proposition}
\begin{proof}
We note that there are $N(n,d)$ such forms, so it suffices to construct
a dual set $\{g_j :  j \in \mathcal I(n,d)\} \subset F_{n,d}$  
so that $[g_j,(i \cdot)^d] = 0$ if $j \neq i$ and  $[g_i,(i \cdot)^d] > 0$. Let
\begin{equation}\label{E:bier}
g_j(x_1,\dots,x_n) = \prod_{k=1}^n \prod_{\ell =0}^{j_k-1} (d x_k - \ell(x_1 +
\cdots + x_n)).
\end{equation}
Each $g_j$ is a product of $\sum_k j_k = d$ linear factors, so $g_j
\in F_{n,d}$. The $(k,\ell)$ factor in \eqref{E:bier} vanishes
 at any $x = i \in \mathcal I(n,d)$ for which $i_k = \ell$. Thus,
$[g_j,(i \cdot)^d] = g_j(i) = 0$ if $i_k \le j_k-1$ for any $k$. Since
$\sum_k i_k = \sum_k j_k$, it follows that $g_j(i) = 0$ if $j \neq i$.
A computation shows that $g_i(i) = d^d\prod_k (i_k !) = d^d d!/c(i)$. 
\end{proof} 
 Prop.\!\! 2.3  implies Prop.\!\! 2.2 directly, upon 
mapping $\mathcal I(n,d)$ linearly into $S$.

\begin{proposition}\cite[p.141]{Re1}
Suppose $B \in \mathcal B_{n,d}$ and there exist $p,q \in 
B$ and $u,v \in \rr^n$ so that $p(u) > 0 > q(v)$. Then $B = F_{n,d}$. 
\end{proposition}
\begin{proof}
By Lemma 1.1(iii), $\pm(\al\cdot)^d \in
B$ for $\al \in \rr^n$, so by Prop.\!\! 2.2, $F_{n,d} \subseteq B$. 
\end{proof}
This is the argument Ellison used in \cite[p.667]{E} to show that
every $p \in  F_{n,u(2v+1)}$ is a sum of $(2v+1)$-st powers of $h_k
\in F_{n,u}$.  

Let $-B = \{ -h: h \in B\}$. It is easy to check
that if $B$ is a blender, then so is $-B$.
\begin{proposition}\cite[p.141]{Re1}
If $B \neq \{0\}, F_{n,d}$ is a blender, then $d=2r$ is even
and for a suitable choice of sign, $Q_{n,2r} \subseteq \pm B
\subseteq P_{n,2r}$.
\end{proposition}
\begin{proof}
If $B \neq \{0\}$, then there exists $p \in B$ and
$a \in \rr^n$ so that $p(a) \neq 0$. If $d$ is odd, then $p(-a) =
-p(a)$, and by Prop.\!\! 2.4, $B = F_{n,d}$. If $d$ is
even, by taking $-B$ if necessary, we may assume that $p(a)
\ge 0$. Thus, if $B \neq F_{n,2r}$, then $\pm B \subseteq
P_{n,2r}$. On the other hand, 
Lemma 1.1 and (P1) imply that $Q_{n,2r} \subseteq \pm B$. 
\end{proof}
Since $Q_{n,2} = P_{n,2}$,  there are no ``interesting''
blenders of quadratic forms. 

The inner product has a useful contravariant property.
\begin{proposition} \cite[p.32]{Re1}
 Suppose $p$, $q\in F_{n,d}$ and $M \in Mat_n(\rr)$. Then 
\begin{equation}\label{E:contra}
[p\circ M,q]=[p,q\circ M^t].
\end{equation}
\end{proposition}

\begin{proof}
By Prop.\!\! 2.2, it suffices to prove \eqref{E:contra}
for $d$-th powers; note that $[p \circ M,q]
= [(\al M\cdot)^d,(\be\cdot)^d] = (\al M \be^t)^d = 
(\al(\be M^t)^t)^d = [(\al \cdot)^d, (\be M^t\cdot)^d] = [p, q \circ M^t]$.
\end{proof}

\begin{proposition}\cite[p.46]{Re1}
If $B$ is a blender, then so is its dual cone $B^*$.
\end{proposition}
\begin{proof}
The dual of a closed convex cone is a closed convex cone, so
(P1) and (P2) are automatic. Suppose $p \in B, q \in
B^*$ and $M \in Mat_n(\rr)$. Since $p\circ M^t \in
B$, we have 
\begin{equation}
[p, q\circ M] = [q \circ M , p] = [q, p\circ M^t] = [p\circ M^t,q] \ge 0,
\end{equation}
and so $q \circ M \in B^*$. This verifies (P3).
\end{proof}

For $i\in\mathcal I(n,d)$, let $D^i = \prod (\frac
{\partial}{\partial x_k})^{i_k}$; let
$f(D) = \sum c(i)a(f;i)D^i$ be the $d$-th order differential operator 
associated to $f \in F_{n,d}$. Since $\frac {\partial}{\partial x_k}$ and\
$\frac {\partial}{\partial x_\ell}$ commute,  $D^iD^j = D^{i+j} =
D^jD^i$ for any $i \in \mathcal I(n,d)$ and $j\in\mathcal I(n,e)$. By
multilinearity,  $(fg)(D) = f(D)g(D) = g(D)f(D)$ for forms $f$ and $g$
of any degree.  
 
\begin{proposition}\cite[p.183]{Re2}
If $i, j \in \mathcal I(n,d)$ and $i \neq j$, then $D^i(x^j) =
0$ and $D^i x^i = \prod_k (i_k)! = d!/c(i)$.  
\end{proposition}
\begin{proof}
We have
\begin{equation}
D^i(x^j) = \prod_{k=1}^n \biggl(\frac {\partial^{^{i_k}}}{\partial x_k^{i_k}}\biggr)
\prod_{k=1}^n x_k^{j_k} =
\prod_{k=1}^n \frac {\partial^{^{i_k}} (x_k^{j_k})}{\partial x_k^{i_k}}.
\end{equation}
If $i_k > j_k$, then the $k$-th factor above is zero. If $i \neq j$,
then this will happen for at least one $k$. Otherwise, $i=j$, and the
$k$-th factor is $i_k!$. 
\end{proof}

We now connect the inner product with differential operators.
 \begin{proposition}\cite[p.184]{Re2}  

\smallskip

\noindent (i) If $p, q \in F_{n,d}$, then  $p(D)q = q(D)p = d![p,q]$.
 
\noindent (ii) If $p, hf \in F_{n,d}$, where $f \in F_{n,k}$ and $h \in
 F_{n,d-k}$,   
then
 \begin{equation}
d![p,hf] = (d-k)![h,f(D)p].
\end{equation}
\end{proposition}
\begin{proof}
For (i), we have by Prop.\!\! 2.8:
\begin{equation}
\begin{gathered}
p(D)q = \sum_{i \in \mathcal I(n,d)} c(i)a(p;i)D^i \biggl(\sum_{j \in
  \mathcal I(n,d)} c(j)a(q;j)x^j\biggr) =
\\ \sum_{i \in \mathcal I(n,d)}  \sum_{j \in \mathcal I(n,d)}
c(i)c(j)a(p;i)a(q;j)D^ix^j 
 = \sum_{i \in \mathcal I(n,d)}  c(i)c(i)a(p;i)a(q;i)D^ix^i 
\\ = \sum_{i \in \mathcal I(n,d)}  c(i)^2a(p;i)a(q;i) \frac {d!}{c(i)} =
d![p,q] = d![q,p] = q(D)p.
\end{gathered}
\end{equation}
\noindent (ii) Two applications of (i) give
\begin{equation}
d![p,hf] = (hf)(D)p = h(D)f(D)p = h(D)(f(D)p) = (d-k)![h,f(D)p].
\end{equation}
\end{proof}
\begin{corollary}
If $p \in F_{n,2r}$, then $Hes(p;u,v) =  2r(2r-1)[p,(u\cdot)^{2r-2}(v\cdot)^2]$.
\end{corollary}
\begin{proof}
Apply Prop.\!\!  2.9 with $h = (u\cdot)^{2r-2}$, $f = (v\cdot)^2$, $d=2r$
and $k=2$. We have
\begin{equation}
f(x_1,\dots,x_n) = (v_1x_1 + \cdots + v_nx_n)^2 \implies
f(D) = \sum_{i=1}^n \sum_{j=1}^n v_iv_j
\frac{\partial^2}{\partial x_i \partial x_j},
\end{equation}
 so that $[h,f(D)p] = Hes(p;u,v)$ by \eqref{E:hes} and Prop.\!\! 2.1(iii). 
\end{proof}

\section{Convergence and duals}

We shall need some tools to prove that certain convex cones are
closed. The first one (see \cite[p.37]{Re1}) is an immediate
consequence of Prop.\!\! 2.2. 

\begin{lemma}
Suppose $S \subset \rr^n$ is bounded and has non-empty interior. Then for $i \in
\mathcal I(n,d)$ and $p \in F_{n,d}$, 
 $|a(p;i)| \le R_{n,d}(i)\cdot\sup\{|p(x)|: x \in S\}$ for some $ R_{n,d}(i)$. 
\end{lemma}
\begin{proof}
Fix $i \in \mathcal I(n,d)$. By Prop.\!\!  2.2, there exist
$\la_k(i)$ and $\al_k \in S$ so that
\begin{equation}\label{E:bound}
x^i = \sum_{k=1}^{N(n,d)} \la_k(i) (\al_k\cdot)^d.
\end{equation}
Taking the inner product of \eqref{E:bound} with $p$, we find that
\begin{equation}
a(p;i) = [p,x^i] =  \sum_{k=1}^{N(n,d)} \la_k(i) [p, (\al_k\cdot)^d ] =
\sum_{k=1}^{N(n,d)} \la_k(i) p(\al_k).
\end{equation}
Now set $R_{n,d}(i) = \sum_k |\la_k(i)|$.
\end{proof}

 We define the norm on $F_{n,d}$ in the usual way, by
\begin{equation}
||p||^2 = [p,p] =  \sum_{i \in \mathcal I(n,d)|} c(i) a(p;i)^2. 
\end{equation}
Given a sequence $(p_m) \in F_{n,d}$, the statement that
$(|a(p_m;i)|)$ is uniformly bounded for all $(i,m)$ is equivalent to the
statement that $(||p_m||)$ is bounded. 

\begin{lemma}
Suppose $(p_{m,r}) \subset F_{n,d}$, $1 \le r \le N$,  and suppose 
that for all $(m,r)$, $|p_{m,r}(u)| \le M$ for  $u \in S$, where $S$
is bounded and
has non-empty interior. Then there exist $p_r \in F_{n,d}$ and $m_k\to\infty$ 
so that simultaneously for each $r$,  $p_{m_k,r} \to p_r$.
\end{lemma}
\begin{proof}
Identify each $p_{m,r}$ with the vector $(a(p_{m,r};i)) \in
\rr^{N(n,d)}$; these vectors are uniformly bounded by Lemma 3.1.
Concatenate them to form a vector $v_m \in \rr^{N*N(n,d)}$. 
By Bolzano-Weierstrass, there is  a convergent
subsequence $(v_{m_k})$. The corresponding subsequences of forms are then
convergent. 
\end{proof}

Even when $(p_m)$ is unbounded, one can still find an interesting subsequence.
\begin{lemma}
Suppose $(p_m) \subset F_{n,d}$ and $||p_m||$ is unbounded. Then
there exists a subsequence $p_{m_k}$ and $\tau_k \to
\infty$ so that $\tau_k^{-1}p_{m_k} \to p$, where $p \neq 0$.
\end{lemma}
\begin{proof}
Let $\mu_m = \max\{|a(p_m;i)|\}$; by hypothesis, $(\mu_m)$ is unbounded.
Take a subsequence on which $\mu_m\to \infty$ and drop the subscripts. Let
$\bar p_m = \mu_m^{-1}p_m$. Then each $\bar p_m$ has at least one
coefficient $a(\bar p_m;i(m)) = \pm 1$. Since $\mathcal I(n,d)$ is finite,
there exists $i_0$ so that there is a subsequence on which $a(\bar
p_{m_k};i_0) = \pm 1$. Taking $-p_{m_k}$ if necessary and dropping the
subscripts, we have $a(\bar p_m;i_0) = 1$ and  $|a(\bar p_m;i)|
\le 1$ for all $(m,i)$. By Lemma 3.2, $(\bar p_m)$ has a convergent subsequence
$\bar p_{m_k} \to p$, and $a(p;i_0) = 1$, so $p \neq 0$.
 Since $\bar p_{m_k} = \mu_{m_k}^{-1} p_{m_k}$, this is the desired
 subsequence.  
\end{proof}

We state without proof a direct implementation of 
Carath\'eodory's Theorem (see 
e.g. \cite[p.27]{Re1}.). It is worth noting that in 1888 (when
Carath\'eodory was 15), Hilbert \cite{Hi} used this
argument with $N(3,6) = 28$ to show that $\Sigma_{3,6}$ is 
closed.
\begin{proposition}[Carath\'eodory's Theorem]
If $r > N(n,d)$, and $h_k \in  F_{n,d}$, then there exist $\la_k
\ge 0$ so that 
\begin{equation}
\sum_{k=1}^r h_k = \sum_{k=1}^{N(n,d)} \la_k h_{n_k}. 
\end{equation} 
\end{proposition}

We use these lemmas to show that if $B_1$ and $B_2$ are
blenders, then so are  $B_1+B_2$ (c.f. \eqref{E:b1+b2}) and $B_1*B_2$ 
(c.f. \eqref{E:b1*b2}). We may assume $B_i \neq 0$.
\begin{theorem}
\ 

\smallskip

\noindent (i) If $B_i \in \mathcal B_{n,2r}$, then $B_1 + B_2 \in
\mathcal B_{n,2r}$.

\noindent (ii) If $B_i \in \mathcal B_{n,2r_i}$ and $r=r_1+r_2$, 
then $B_1*B_2 \in \mathcal B_{n,2r}$. 
\end{theorem}
\begin{proof}
In each case, (P1) is automatic, and
since $(p_1+p_2) \circ M = p_1 \circ M + p_2 \circ M$ and $(p_1p_2)
\circ M = (p_1 \circ M)( p_2 \circ M)$, (P3) is verified. The issue is
(P2). 

Suppose $B_i \in \mathcal B_{n,2r}$ have opposite ``sign'', say
$B_1 \subset P_{n,2r}$ and $B_2 \subset -P_{n,2r}$. Then Prop.\!\! 2.4
implies that $B_1 + B_2 = F_{n,2r}$. Otherwise, we may assume 
that $B_i \subset P_{n,2r_i}$.
Suppose $p_{i,m} \in B_i$ and $p_{1,m} + p_{2,m} = p_m\to
p$.  Let $S$ be the unit ball in $\rr^n$.
If $\sup\{p(u) : u \in S\} = T$, then for $m \ge m_0$,
  $\sup\{p_m(u) : u \in S\} \le T+1$, and since $p_{i,m}$ is psd, it
follows that   $\sup\{p_{i,m}(u) : u \in S\} \le T+1$ as well. By
  Lemma 3.2, there is a common subsequence so that $p_{i,m_k} \to p_i
  \in B_i$, hence $p = \lim p_{m_k} = p_1+p_2 \in B_1 + B_2$.

Suppose now   $B_i \in \mathcal B_{n,2r_i}$, and by taking $\pm B_i$,
assume $B_i \subset P_{n,2r_i}$.
By Prop.\!\!  3.4, a sum such as \eqref{E:b1*b2} can be compressed
into one in which 
$s \le N(n,2r)$. Write
\begin{equation}
p_m =  \sum_{k=1}^{N(n,2r)}  p_{1,k,m}p_{2,k,m},\qquad p_{i,k,m} \in B_i,
\end{equation}
and suppose $p_m \to p$.
As above, since $p$ is bounded on $S$, so is the sequence $(p_m)$,
and since each $p_{i,k,m}$ is psd, it follows that the
sequence $(p_{1,k,m}p_{2,k,m})$ is bounded on $S$, and hence by
Lemma 3.2,  a subsequence of  $(p_{1,k,m}p_{2,k,m}) \to p_k$ for some $p_k \in
P_{n,2r}$.  We need to show that $p_k$ can be written as a product
$q_{1,k}q_{2,k}$, where $q_{i,k} \in B_i$. 
A complication is that the given sequence of factors might not both
converge (e.g.\!\! if $p_{1,k,m} = m q_{1,k}$ and $p_{2,k,m} = m^{-1}
q_{2,k}$), so we need to normalize.

First observe that if $p_k=0$, we are done. Otherwise, choose $v \in \rr^n$
 so that $p_k(v) =1$. Since $p_{1,k,m}(v)p_{2,k,m}(v) \to 1$, 
$p_{1,k,m}(v)p_{2,k,m}(v) > 0$ for $m \ge m_0$. Drop the first
$m_0$ terms and define  
\begin{equation}
q_{1,k,m}(x) = \frac {p_{1,k,m}(x)}{p_{1,k,m}(v)} \in B_1, \qquad
q_{2,k,m}(x) = \frac {p_{2,k,m}(x)}{p_{2,k,m}(v)} \in B_2.
\end{equation}
Then $(q_{1,k,m}q_{2,k,m}) \to p_k$ and $q_{i,k,m}(v) = 1$. 

If each $(||q_{i,k,m}||)$ is bounded, then by
Lemma 3.2, there are convergent subsequences $q_{i,k,m} \to q_{i,k} \in
B_i$ and $p_k = q_{1,k}q_{2,k}$ as desired. 

Suppose $(||q_{1,k,m}||)$ is unbounded and  $(||q_{2,k,m}||)$ is
bounded.  Taking the common convergent subsequences from
Lemmas 3.2 and 3.3, and dropping subscripts, we have $\tau_m \to
\infty$ and  $q_{1,k,m} = \tau_m \bar q_{1,k,m}$ so that $\bar
q_{1,k,m} \to \bar q_{1,k} \in B_1$ (where $\bar q_{1,k} \neq 0$) and $q_{2,k,m} \to 
q_{2,k} \in B_2$, where  $q_{2,k}(v) = \lim q_{2,k,m}(v) = 1$,
so $q_{2,k} \neq 0$. But now
\begin{equation}
0 = \lim_{m \to \infty} \tau_m^{-1} q_{1,k,m}q_{2,k,m} = \lim_{m \to \infty} \bar
q_{1,k,m}q_{2,k,m} = \bar q_1q_2,
\end{equation}
a contradiction. If both $(||q_{i,k,m}||)$'s are unbounded, we
can write  $q_{2,k,m} = \nu_m \bar q_{2,k,m}$ with $\nu_m \to
\infty$ and $\bar q_{2,k,m} \to \bar q_{2,k}\neq 0$ and derive a similar
contradiction.  It follows that the first case holds for each $k$ and
so $B_1*B_2$ satisfies (P2).
\end{proof}
The following theorem was announced in \cite[p.47]{Re1}, but the proof
was not given.
\begin{theorem}
If $uv = r$, then $W_{n,(u,2v)}$ is a blender.
\end{theorem}
\begin{proof}
As we have already seen, (P1) and (P3) are immediate. Suppose $p_m \in
W_{n,(u,2v)}$ and $p_m \to p$.  Prop.\!\! 3.4 says that we can write
\begin{equation}
p_m = \sum_{k=1}^{N(n,2r)}h_{k,m}^{2v}, \qquad h_{k,m} \in F_{n,u}.
\end{equation}
As before, $p$ (and so $(p_m)$)  is bounded on $S$, and the
summands are psd so   $(h_{k,m}^{2v})$ and thus also $(|h_{k,m}|) =
((h_{k,m}^{2v})^{1/(2v)})$ are bounded  on $S$.  Taking a convergent
subsequence, suppose $(h_{k,m}) \to
h_k$. Then $(h_{k,m}^{2v}) \to h_k^{2v}$. Taking a common subsequence
for each of the $N(n,2r)$ summands, we see that $p \in W_{n,(u,2v)}$.
\end{proof}

In particular, $\Sigma_{n,2r}$ and $Q_{n,2r}$ are blenders; see \cite[p.46]{Re1}.
\begin{theorem}
If $\sum_iu_iv_i = 2r$, then
$W_{n,\{(u_1,2v_1),...., (u_m,2v_m)\}} \in \mathcal B_{n,2r}$.
\end{theorem}
\begin{proof}
Note that $W_{n,\{(u_1,2v_1),...., (u_m,2v_m)\}} = W_{n,(u_1,2v_1)} * \cdots
*  W_{n,(u_m,2v_m)}$.
\end{proof}
 
\begin{proposition}\cite[p.38]{Re1}
$P_{n,2r}$ and $Q_{n,2r}$ are dual blenders.
\end{proposition}
\begin{proof}
We have $p \in Q_{n,2r}^*$ if and only if $p \in F_{n,2r}$ and, whenever $\la_k
\ge 0$ and $\al_k \in \rr^n$, 
\begin{equation}
0 \le \left[ p, \sum_{k=1}^r \la_k (\al_k\cdot)^{2r} \right] =
\sum_{k=1}^r \la_k p(\al_k).
\end{equation}
This is true iff $p(\al) \ge 0$ for all $\al \in \rr^n$;
that is, iff $p \in P_{n,2r}$.
\end{proof}
It was a commonplace by the time of \cite{Hi} that
$P_{n,2r} = \Sigma_{n,2r}$ when $n=2$ or $2r=2$.  Hilbert proved there that
$P_{3,4} = \Sigma_{3,4}$ and that strict inclusion is true for other
$(n,2r)$ (see \cite{Re3}.) 
We say that $p \in P_{n,2r}$ is 
{\it positive definite} or {\it pd} if $p(u) = 0$ only for $u=0$.
It follows that $p \in int(P_{n,2r})$ if and only if $p$ is pd.

Blenders are cousins of orbitopes. An {\it orbitope}
is the convex hull of an orbit of a compact algebraic group $G$ acting
linearly on a real vector space; see  \cite[p.1]{SSS}. The key
differences from blenders are that it is a single orbit, and that $G$ is
compact. One object which is both a blender and an orbitope is
$Q_{n,2r}$, which  is named $\mathcal V_{n,2r}$  (and called the {\it 
  Veronese orbitope}) in \cite{SSS}. 

The duals of the Waring blenders can be explicitly given.
\begin{proposition}\cite[p.47]{Re1}
Given $p \in F_{n,2uv}$, define the form $H_p(t) \in F_{N(n,u),2v}$, in
variables $\{t(\ell)\}$ indexed by $\{\ell \in \mathcal I(n,u)\}$, by
\begin{equation}\label{E:wardual}
H_p(\{t(\ell_j)\}) =  \sum_{\ell_1 \in \mathcal I(n,u)}\cdots
\sum_{\ell_{2v}  \in \mathcal I(n,u)} a(p;\ell_1 + \cdots + \ell_{2v})t(\ell_1)\cdots
t(\ell_{2v}). 
\end{equation}
Then $p \in W_{n,(u,2v)}^*$ if and only if $H_p \in P_{N(n,u),2v}$.
\end{proposition}
\begin{proof}
We have $p \in W_{n,(u,v)}^*$ if and only if, for every form $g \in
F_{n,u}$, $[p,g^{2v}] \ge 0$. Writing
$g \in F_{n,u}$ with coefficients $\{t(\ell): \ell \in \mathcal I(n,u)\}$,
we have:
\begin{equation}
\begin{gathered}
g(x) = \sum_{\ell \in \mathcal I(n,u)} t(\ell) x^\ell \implies \\
g^{2v}(x) =  \sum_{\ell_1 \in \mathcal I(n,u)}\cdots\sum_{\ell_{2v} \in \mathcal I(n,u)}
t(\ell_1)\cdots t(\ell_{2v}) x^{\ell_1 + \cdots+ \ell_{2v}}.
\end{gathered}
\end{equation}
It follows from \eqref{E:ip} and \eqref{E:wardual} that $[p,g^v] =
H_p(t(\ell))$. 
\end{proof}

If $v=1$, then $\mathcal I(n,1) = \{e_i\}$ and, upon writing $t(e_i) =
y_i$, $H_p(y_1,\dots,y_n) = p(y)$; we recover $Q_{n,2r}^* = P_{n,2r}$. 
If $u=1$, then $H_p$ becomes the classical catalecticant and 
\begin{equation}\label{E:cata}
p \in \Sigma_{2,2r}^* 
\iff H_p(t)  = \sum_{i \in \mathcal I(n,r)}\sum_{j \in \mathcal I(n,r)}
a(p;i+j)t(\ell_i)t(\ell_j)\ \text{ is $psd$}. 
\end{equation}
This shows that $\Sigma_{n,2r}$ is a spectrahedron (see \cite[p.27]{SSS}).

\begin{theorem}
If $\sum v_i = r$, then $W_{2,\{(1,2v_1),\dots,(1,2v_m)\}} = P_{2,2r}$
iff $m=r$ and $v_i=1$. 
\end{theorem}
\begin{proof}
If $p \in P_{2,2r} = \Sigma_{2,2r}$, then $p = f_1^2 + f_2^2$, where $f_i
\in F_{2,r}$. Factor $\pm f_i$ into a product of linear and pd
quadratic factors (themselves a sum of two squares):
\begin{equation}
f_i = \prod_j \ell_{1,j} \prod_k (\ell_{2,k}^2 + \ell_{3,k}^2).
\end{equation} 
Then, using \eqref{E:h22} and expanding the product below, we see that
\begin{equation}
f_i^2 = \prod_j \ell_{1,j}^2 \prod_k \bigl((\ell_{2,k}^2 -
  \ell_{3,k}^2)^2 + (2\ell_{2,k}\ell_{3,k})^2 \bigr) \in W_{2,\{(1,2),\dots,(1,2)\}}.
\end{equation}
The converse inclusion follows from Prop.\!\! 2.5.

Suppose $m < r$ and suppose
\begin{equation}
\prod_{\ell=1}^r (x - \ell y)^2 = \sum_{k=1}^s  h_{k,1}^{2v_1}\cdots
h_{k,m}^{2v_m} ,\quad  h_{k,i}(x,y) = \al_{k,i}x + \be_{k,i}y  \in F_{2,1}.
\end{equation}
Then for each $k$, we have 
\begin{equation}
\prod_{\ell=1}^r (x - \ell y)\ \bigg \vert \ \prod_{i=1}^m (\al_{k,i}x + \be_{k,i}y);
\end{equation}
since $m < r$, the right-hand side is 0, and we have a contradiction.
\end{proof} 

Finally, we have a simple expression for $K_{n,2r}^*$; this
seems to be implicit in  \cite{B}.
\begin{theorem}
$K_{n,2r}$ and $W_{n,\{(1,2r-2),(1,2)\}}$ are dual blenders.
\end{theorem}
\begin{proof}
By Corollary 2.10 and the Hessian definition, $p$ is convex if and
only if $0 \le Hes(p;u,v) = 2r(2r-1)[p, (u\cdot)^{2r-2}(v\cdot )^2]$
for all $u,v \in \rr^n$. 
\end{proof}

 It follows from Theorems 3.10 and 3.11 that
 $K_{2,4}^* = W_{2,\{(1,2),(1,2)\}} = P_{2,4}$, so $K_{2,4} =
 Q_{2,4}$. For $r \ge 3$,  $K_{2,2r}^* = W_{2,\{(1,2r-2),(1,2)\}}
 \subsetneq P_{2,4}$, so $K_{2,2r} \supsetneq Q_{2,2r}$.  We 
return to this topic in section six.

\section{$K_{n,2r}$: convex forms}
In this section, we prove some general results for $K_{n,2r}$. 
Since $p \in K_{n,2r}$ if and only if $Hes(p;u,v)$ is psd and
$Hes(p;u,u) = 2r(2r-1)p(u)$, we get an alternative proof that
$K_{2,2r} \subseteq P_{n,2r}$. We also know from Theorem 3.11 that $p
\in int(K_{n,2r})$ if and only if $[p,q] > 0$ for $0 \neq q \in
W_{n,\{(1,2r-2),(1,2)\}}$; accordingly, $int(K_{n,2r})$ is the set of $p
\in K_{2,2r}$ so that $Hes(p;u,v)$ is positive definite as a
bihomogeneous form in the variables $u \in \rr^n$ and $v \in
\rr^n$. Equivalently, $p \in K_{n.2r}$ is in $\partial(K_{n,2r})$ if and only if
there exist $u_0\neq 0, v_0 \neq 0$ such that $Hes(p;u_0,v_0) = 0$. 

Although the  psd and sos  properties are preserved under homogenization and
dehomogenization, this is not true for convexity. For example, $t^2-1$
is a convex polynomial which cannot be homogenized to a convex form,
because it is not definite. As a pd polynomial in one variable, 
 $t^4 + 12 t^2 + 1$ is convex, but if $p(x,y) = x^4 +
 12x^2y^2 + y^4$, then $Hes(p;(1,1),(v_1,v_2)) = 36v_1^2 +
 96v_1v_2 + 36v_2^2$ is not psd, so $p$ is not convex. 

\begin{proposition}
If $p \in K_{2,2r}$, then there is a pd form $q$ in $\le n$ variables
and $\bar p \sim p$ such that $\bar p(x) = q(x_k,\dots, x_n)$. 
\end{proposition}
\begin{proof}
If $p$ is pd, there is nothing to prove. Otherwise, we can assume that
$p \sim \bar p$, where $\bar p$ is convex and 
$\bar p(e_1) = 0$. We shall show that $\bar p = \bar
p(x_2,\dots,x_n)$. Repeated application of this argument then proves the result.

 Suppose otherwise that $x_1$ appears in a term of $\bar p$ and let
$m \ge 1$ be the largest such power of $x_1$; write 
the associated terms in $\bar p$ as $x_1^mh(x_2,...,x_n)$. After
an additional invertible  linear change involving $(x_2,\dots,x_n)$,
we may assume that 
one of these terms is $x_1^mx_2^{2r-m}$. We then  have 
\begin{equation}
\bar p(x_1,x_2,0,\dots,0) = x_1^m x_2^{2r-m} + \text{lower order terms in $x_1$}
\end{equation} 
which implies that
\begin{equation}\label{E:binhess}
\begin{gathered}
\frac{\partial^2 \bar p}{\partial x_1^2}\frac{\partial^2 \bar
  p}{\partial x_2^2}
- \left(\frac{\partial^2 \bar p}{\partial x_1\partial x_2}\right)^2 =  \\
 -(2r-1)m(2r-m) x_1^{2m-2}x_2^{4r-2m-2}  + \text{lower order terms in $x_1$}.
\end{gathered}
\end{equation}
Since $r \ge 1$ and $1 \le m \le 2r-1$, \eqref{E:binhess} cannot be
psd, and this contradiction shows that $x_1$ does not occur in $\bar p$.
\end{proof}
\begin{corollary}
There do not exist $B_i \in \mathcal B_{n,2r_i}$, $r_i \ge 1$, so that
$K_{n,2r_1+2r_2} = B_1 * B_2$.
\end{corollary}
\begin{proof}
It follows from  Prop.\!\!  2.5 that $x_i^{2r_i} \in B_i$, hence
$x_1^{2r_1}x_2^{2r_2} \in B_1*B_2$, but by Prop.\!\! 4.1, this form
is not convex. 
\end{proof}
The next theorem connects $K_{n,2r}$ with the blender $N_{n,2r}$ defined
in \cite[p.119-120]{Re1}. Let $E = <\!\! e_1,\dots,e_n\! \!>$ be a real
$n$-dimensional vector space. We say that $f$ is a {\it
  norm-function}  on $E$ if, after defining 
\begin{equation}\label{E:nf}
||x_1e_1+ \dots + x_ne_n|| = f(x_1,\dots,x_n),
\end{equation}
the pair $(E,||\cdot||)$ is a Banach space. Let
\begin{equation}\label{E:nnd}
N_{n,d}:= \{p \in F_{n,d}: p^{1/d} \text{ is a norm function} \}.
\end{equation}
A necessary condition is that $f = p^{1/d} \ge 0$, hence $d=2r$ is
even and $p \in P_{n,2r}$.
For example, if $p(x) = \sum_k x_k^2$, then \eqref{E:nf} with $f=p^{1/2}$
gives $\rr^n$ with the Euclidean norm. If $(E,||\cdot||)$ is isometric to a
subspace of some $L_{2r}(X,\mu)$, then $f^{2r} \in Q_{n,2r}$.
The following theorem was proved in the author's thesis; see \cite{Re0,Re00}.

\begin{proposition}\cite[Thm.1]{Re00}
If $p \in P_{n,2r}$, then $p \in N_{n,2r}$  iff for all $u,v \in \rr^n$,
$p(u_1+tv_1,\dots,u_n+tv_n)^{1/d}$ is a convex function of $t$. 

\end{proposition}

It is not obvious  that $N_{n,2r}$ is a blender; in fact,
$N_{n,2r} = K_{n,2r}$! The connection is a proposition  whose
provenance is unclear. It appears in Rockafellar's monograph
\cite[Cor.15.3.1]{Ro}, where it is attributed to Lorch \cite{Lo}, although the
derivation is not transparent. 
V. I. Dmitriev (see section 6) attributes the result to an
observation by his advisor S. G.  Krein in 1969. Note below that $q$ is
{\it not} homogeneous. 

\begin{proposition}
Suppose $p \in P_{n,2r}$ and $p(1,0,...,0) > 0$. Let 
\begin{equation}
q(x_2,\dots,x_n) = p(1,x_2,\dots,x_n). 
\end{equation}
Then $p \in K_{n,2r}$ if and only if $q^{1/(2r)}(x_2,\dots,x_n)$ is
convex.
\end{proposition}

\begin{corollary}
$K_{n,2r} = N_{n,2r}$.
\end{corollary}
\begin{proof}[Proof of Prop.\!\! 4.4]
A function is convex if and only if it is convex when restricted
to all two-dimensional subspaces. 
Consider all $a \in \rr^N$ with $a_1 = 1$. Suppose we can show that 
$Hes(p;a,u)$ is psd in $u$ if and only if $q^{1/(2r)}$ is
convex at $(a_2,\dots,a_n)$. By homogeneity, this
occurs if and only if $Hes(p;a,u)$ is psd in $u$ for every $a$ with $a_1
\neq 0$ and by continuity, this holds if and only
if $Hes(p;a,u)$ is psd for all $a,u$. Thus, it suffices to
set $a_1 = 1$ and prove the equivalence pointwise. 

 Fix $(a_2,\dots,a_n)$ and let
\begin{equation}
\begin{gathered}
\tilde p(x_1,x_2\dots, x_n) = p(x_1, x_2 + a_2 x_1, \dots, x_n + a_nx_1),\\
\tilde q(x_2,\dots,x_n) = \tilde p(1,x_2,\dots x_n) = q(x_2+a_2,\dots,x_n+a_n)
\end{gathered}
\end{equation}
Then $p$ and $q^{1/(2r)}$ are convex at $a$ and $(a_2,\dots,a_n)$  iff
$\tilde p$ and 
$\tilde q$ are convex at $e_1$ and 0, and we can drop the tildes and
assume that  
$a_k = 0$ for $k \ge 2$, so $a = e_1$. 
Since it suffices to look at all two-dimensional
subspaces containing $e_1$, we make one more change of variables in
$(x_2,\dots,x_n)$, and assume  this 
subspace is $\{(x_1,x_2,0,\dots,0)\}$.

Suppose now that
\begin{equation}
h(x_1,x_2) = p(x_1,x_2,0,\dots,0) = a_0 x_1^{2r} + \binom {2r}1 a_1
x_1^{2r-1}x_2 + \binom {2r}2 a_2x_1^{2r-2}x_2^2 + \dots. 
\end{equation}
Then
\begin{equation}
Hes(h;(1,0),(v_1,v_2)) = 2r(2r-1)(a_0 v_1^2 + 2a_1 v_1v_2 + a_2 v_2^2), 
\end{equation} 
and since $a_0 = p(e_1) > 0$, this is psd iff $a_0a_2 \ge  a_1^2$. On
the other hand,  
\begin{equation}
q(t) = p(1,t) = a_0 + \binom {2r}1 a_1 t + \binom {2r}2 a_2 t^2 + \dots
\end{equation}
and a routine computation shows that
\begin{equation}
(q^{(1/(2r))})''(0) = (2r-1)a_0^{-2+1/(2r)}(a_0a_2 - a_1^2).
\end{equation}
Thus the two conditions hold simultanously.
\end{proof}

A more complicated proof computes the Hessian of $p$, uses the Euler
PDE ($2rp = \sum x_i \frac{\partial p}{\partial x_i}$ and 
$(2r-1)\frac{\partial p}{\partial x_i} = \sum x_j\frac{\partial^2
  p}{\partial x_i\partial x_j} $) to replace partials involving $x_1$
with partials involving only the other variables. 
The discriminant  of this Hessian with respect to $u_1$ (after a
change of variables) becomes a positive multiple of the Hessian of $q^{1/(2r)}$.

We conclude this section with a peculiar result which implies that
every pd form is, in a computable way, the restriction of a convex
form on $S^{n-1}$.  

\begin{theorem}
Suppose $p \in P_{n,2r}$ is pd, and let $p_N:= (\sum_j x_j^2)^N p$.
Then there exists $N$ so that $p_N \in K_{n,2r+2N}$.
\end{theorem}

\begin{proof} 
Since $p$ is pd, it is bounded away from 0 on $S^{n-1}$ and so there are
uniform upper bounds $T$ for $ |p(x)^{-1}\nabla_u(p)(x)|$ and   $U$
for $|p(x)^{-1} \nabla^2_u(p)(x)|$,  for $x, u \in
S^{n-1}$. Since $\sum x_i^2$ is rotation-invariant, once again it suffices to
show that $p_N$ is  convex at $(1,0,\dots,0)$, given $x_3 = \cdots = x_n
= 0$. We claim that if $N > (T^2 + U)/2$, then $p_N$
 is convex. By Prop.\!\! 4.4, it suffices to show that
$p^{1/(2N+2r)}_N(1,t,0,\dots,0)$ is convex at $t=0$. 
Writing down the relevant Taylor series, this becomes
 \begin{equation}
 (1 + t^2)^{N/(2N+2r)} (1 + \al t + \tfrac 12 \be t^2 + \dots )^{1/(2N+2r)},
 \end{equation}
 where $|\al| \le T$ and $|\be|\le U$. By expanding the product, a
 standard computation shows that the second derivative at $t=0$ is
 \begin{equation}
\frac {N}{N+r} + \frac 1{2N+2r}\cdot b - \frac
{2N+2r-1}{(2N+2r)^2}\cdot a^2
\ge  \frac 1{2N+2r} \left(2N - U - T^2 \right) \ge 0.
 \end{equation}
\end{proof}

Greg Blekherman pointed out to the author's chagrin in Banff that
Theorem 4.6
 follows from \cite[Thm.3.12]{R2}: if $p$ is pd, then there exists $N$
so that $p_N \in Q_{n,2r+2N}$. This was used in \cite{R2} to show that
$P_N\in \Sigma_{n,2r+2N}$; it also implies that $p \in K_{n,2r+2N}$. 
The proof of  \cite[Thm.3.12]{R2} is much less elementary.

We conclude this section with a computational illustration of the
proof of Theorem 4.6.
If $a > 0$, then $x^2 + a y^2$ is convex, but if $r \ge 1$ and
 $(x^2 + y^2)^r(x^2 + a x^2) \in K_{2,2r+2}$ for all $a>0$, then  by (P2),
  $x^2(x^2 + y^2)^r$ would be convex, violating Prop.\!\! 4.1. 

\begin{theorem}
\begin{equation}
(x^2 + y^2)^r(x^2 + a x^2) \in K_{2,2r+2} \iff a + 1/a \le 8r + 18 + 8/r.
\end{equation}
\end{theorem}
\begin{proof}
Let $p(x,y) = (x^2 + y^2)^r(x^2 + a x^2)$. A computation shows that
\begin{equation}  
\begin{gathered}
\frac{\partial^2p}{\partial x^2}\frac{\partial^2p}{\partial y^2} -
 \left(\frac{\partial^2p}{\partial x\partial y}\right)^2 =
4(2r+1)(x^2+y^2)^{2r-2} q(x,y), \quad \text{where} \quad 
q(x,y) = \\ (1 + r) (a + r)x^4 + ( 2 a - r + 6 a r - a^2 r + 2 a
r^2)x^2y^2 +  a (1 + r) (1 + a r)y^4.
\end{gathered}
\end{equation}
Another computation shows that
\begin{equation}\label{E:mess}  
\begin{gathered}
4(1+r)(a+r)q(x,y)\\ = (2(1+r)(a+r)x^2 + ( 2 a - r + 6 a r - a^2 r + 2 a
r^2)y^2)^2 \\ + 4a r^2(a-1)^2\bigl((8r + 18 + 8/r)-(a+1/a)\bigr)y^4.
\end{gathered}
\end{equation}
If $a + 1/a \le 8r + 18 + 8/r$, then \eqref{E:mess} shows that $q$ 
 is psd. Suppose $a + 1/a >  8r + 18 + 8/r$. Observe that
$ 2 a - r + 6 a r - a^2 r + 2 ar^2 \ge
0$ if and only if $(a + 1/a) \le  2r + 6 + 2/r$, so in this case, $ 2
a - r + 6 a r - a^2 r + 2 ar^2 < 0$ and we can choose $(x,y) = (x_0,y_0) \neq
(0,0)$ to make the first square in \eqref{E:mess}  equal to zero. It then
follows that $4(1+r)(a+r)q(x_0,y_0)< 0$.
\end{proof}
In particular, $(x^2  + y^2)(x^2 + a y^2) \in K_{2,4} \iff
 17 - 12 \sqrt 2 \le a \le 17 +  12 \sqrt 2$.

\section{$\mathcal B_{2,4}$: binary quartic blenders}
 
In view of Prop.\!\! 2.5, the simplest non-trivial opportunity to
classify  blenders comes with the binary quartics. Throughout this section, we
 choose a sign for $\pm B \in \mathcal B_{2,4}$ and assume that $B
\subset P_{2,4}$. We shall show that $\mathcal B_{2,4}$
is a one-parameter nested family of blenders increasing from $Q_{2,4}$ to
$P_{2,4}$.  It is also convenient to let $Z_{2,4}$ denote the set of $p \in
P_{2,4}$ which are neither pd not a 4th power; if $p \in Z_{2,4}$,
then $p = \ell^2 h$, where $\ell$ is linear and $h$ is a psd quadratic
form relatively prime to $\ell$.

\begin{lemma}
If  $B \in \mathcal B_{2,4}$ and $0 \neq p \in B \cap Z_{2,4}$, then  $B =P_{2,4}$.
\end{lemma}
\begin{proof}
We have $p \sim q$, where 
$q(x,y) = x^2(ax^2 + 2bxy + cy^2) \in B$,  $ac - b^2\ge 0$ and $c > 0$. But
\begin{equation}
 x^2(ax^2 + 2bxy + cy^2) =
 x^2\bigl( \bigl(\tfrac{ac - b^2}c\bigr) x^2 + c\bigl(\tfrac bc x + y
 \bigr)^2\bigr) \sim x^2(d x^2 + c y^2),
\end{equation}
and $d \ge 0$. Next, $(x,y) \mapsto (\ep x, \ep^{-1}y)$
shows that $\ep^2 dx^4 + c x^2y^2 \in B$, so  $x^2y^2 \in  B$ by (P2) and
$\ell_1^2\ell_2^2 \in B$ by (P3). Thus, $W_{2,\{(1,2), (1,2)\}} = P_{2,4}
\subseteq B$ by Theorem 3.10.
\end{proof}

This lemma illustrates one difference between
blenders and orbitopes. If $G = SO(2)$ and $p(x,y) = x^2(x^2+y^2)$,
then the image of $p$ under the action of $G$ will be $\{(\cos t x + \sin t
y)^2(x^2+y^2)\}$, so even taking scalar multiples into account, the
convex hull will not contain the 4th powers or the square of an
indefinite quadratic.

A binary quartic of particular importance is
\begin{equation}\label{E:flam}
f_{\la}(x,y) := x^4 + 6\la x^2y^2 + y^4;
\end{equation}
we also define
\begin{equation}
g_{\la}(x,y):= f_{\la}(x+y,x-y) = (2+6\la)x^4 + (12 - 12\la)x^2y^2 + (2+6\la)y^4. 
\end{equation}
We shall need two special fractional linear transformations. Let
\begin{equation}\label{E:TU}
T(z):= \frac{1-z}{1+3z}, \qquad U(z) := - \frac{1+3z}{3-3z}.
\end{equation}
It follows from \eqref{E:flam} that $g_{\la} = (2+6\la)f_{T(\la)}$, hence for
$\la \neq -\frac 13$, $f_{\la} \sim f_{T(\la)}$. Note that  $T(T(z))
= z$, $T(0) = 1$, $T(\frac 13) = \frac 13$, and 
$T(-\frac 13) = \infty$ (corresponding to $(x^2-y^2)^2 \sim x^2y^2$);
$T$ gives a 1-1
decreasing map between $[\frac 13,\infty)$ and $(-\frac 13,\frac 13]$. 
We also have
\begin{equation}\label{E:apo}
[f_{\la},g_{\mu}] = (2+6\mu) + \la(12-12\mu) + (2+6\mu) =
4(1+3\la+3\mu-3\la\mu).
\end{equation}
Note that  $U(U(z)) = z$, $U(0) = -\tfrac 13$, $U$ gives a
1-1 decreasing map from $[-\frac 13,0]$ to itself, and
\begin{equation}\label{E:fglam}
[f_{\la},g_{U(\la)+\tau}] = 12(1-\la)\tau. 
\end{equation}
It follows from \eqref{E:fglam} that $[f_{\la},g_{U(\la)}] = 0$, and
if $\la < 1$ and $\mu < U(\la)$, then $[f_{\la},g_{\mu}] < 0$.

It is easy to see directly from \eqref{E:flam} that $f_{\la}$ is psd iff $\la
\in [-\frac 13,\infty)$, and pd iff $\la \in (-\frac 13,\infty)$, and
from (P3) that, if $B \in \mathcal B_{2,4}$, then
\begin{equation}
f_{\la} \in B \iff f_{T(\la)} \in B.
\end{equation}
By (P1), if $-\frac 13 < \la \le \frac 13$, then $f_{\la}
\in B$ implies that $f_{\mu} \in B$ for $\mu \in [\la,T(\la)]$.  

It is classically known that a ``general'' binary quartic can be put
into the shape $f_{\la}$ for some $\la$ after an invertible linear
transformation. However there is no guarantee that the coefficients of
the transformation are real, and the result is not universal: $x^4
\not \sim f_{\la}$. The following first appeared in \cite[Thm.6]{PR}.

\begin{proposition}
 If $p \in P_{2,4}$ is pd, then $p \sim  f_{\la}$ for some $\la \in
(-\frac 13,\frac 13]$.
\end{proposition}
\begin{proof}
Suppose first $p = g^2$. Then $g$ is pd, so $g
\sim x^2+y^2$ and $p \sim f_{\frac 13}$.  

If $p$ is not a perfect square, then it is a product of two pd
quadratic forms; we may assume that $p(x,y) = (x^2+y^2)q(x,y)$, with
\begin{equation}
q(x,y) = ax^2 + 2bxy +cy^2.
\end{equation}
A ``rotation of axes'' fixes $x^2+y^2$ and takes $q$ into
$d x^2 + ey^2$ with $d,e > 0$, $d \neq e$, so $p
\sim(x^2+y^2)(dx^2+ey^2)$. Now, 
$(x,y) \mapsto (d^{-1/4}x,e^{-1/4}y)$ gives $p
\sim f_{\mu}$, where $\mu = \frac 16(\ga + \ga^{-1}) > \frac 13$ for $\ga =
\sqrt{d/e}\neq 1$. Thus, $p \sim f_{T(\mu)}$ where $T(\mu) \in
(-\frac 13,\frac 13)$.  
\end{proof}

We need some results from classical algebraic geometry.
Suppose
\begin{equation}
p(x,y) = \sum_{k=0}^4\binom 4k a_k(p) x^{4-k}y^k.
\end{equation} 
The two ``fundamental invariants'' of $p$ are
\begin{equation}
\begin{gathered}
I(p) = a_0(p)a_4(p) - 4a_1(p)a_3(p)+3a_2(p)^2, \\ 
J(p) = \det
\begin{vmatrix}
a_0(p) & a_1(p) & a_2(p) \\
a_1(p) & a_2(p) & a_3(p)\\
a_2(p)& a_3(p) & a_4(p)
\end{vmatrix}.
\end{gathered}
\end{equation} 
(Note $J(p)$  is the determinant of the catalecticant matrix $H_p$.)
We have $I(f_{\la}) = 1 + 3\la^2$ and $J(f_{\la})= \la - \la^3$, but
$I(x^4)=J(x^4)=0$.  
It follows from Prop.\!\! 5.2 that if $p$ is pd, then $I(p) > 0$.
It is  easily
checked that if $q(x,y) = p(ax + by,cx+dy)$, then 
\begin{equation}\label{E:inv}
I(q) = (ad-bc)^4 I(p), \qquad J(q) = (ad - bc)^6 J(p).
\end{equation}
Let
\begin{equation}\label{E:inv2}
K(p) := \frac {J(p)}{I(p)^{3/2}}. 
\end{equation} 
It follows from \eqref{E:inv} and \eqref{E:inv2} that, if $p \sim q$,
then $K(q) = K(p)$. 
In particular, 
\begin{equation}
p \sim f_{\la} \implies K(p) = K(f_{\la}) = \phi(\la):=
\frac{\la-\la^3}{(1+3\la^2)^{3/2}}.    
\end{equation}
\begin{lemma}
If $p$ is pd, then $p \sim f_{\la}$, where $\la$ is the unique
solution  in $(-\frac 13, \frac 13]$ to $K(p) = \phi(\la)$. If $p \in
Z_{2,4}$, then $K(p) = \phi(-\frac 13)$. 
\end{lemma}
\begin{proof}
By Proposition 5.2, $p \sim f_{\la}$ for some $\la \in (-\frac 13,
\frac 13]$. A routine computation shows that $f'(\la) =
(1-9\la^2)(1+3\la^2)^{-5/2}$ is positive on $(-\frac 13, \frac 13)$,
hence $\phi$ is strictly increasing. By Lemma 5.1, if $p \in Z_{2,4}$,
then $p \sim q$, where $q(x,y) = dx^4 + 6e x^2y^2$ for some $e > 0$. 
Since $I(q) =3e^2$ and $J(q) = -e^3$, $K(p) = K(q) = 3^{-3/2} = \phi(-\frac
13)$. 
\end{proof}

\begin{theorem}
Suppose $r,s \in [-\frac 13,0]$, and suppose $1+3r+3s-3rs = 0$; that
is, $s = U(r)$. If $ p
\in [[f_r]]$ and $q \in [[f_s]]$, then $[p,q] \ge 0$.
\end{theorem}
\begin{proof}
Suppose $p = f_r \circ M_1$ and $q = f_s \circ M_2$. Then
\begin{equation}
[p,q] = [f_r\circ M_1,f_s \circ M_2] = [f_r, f_s \circ M_2M_1^t],
\end{equation} 
hence it suffices to show that for all $a,b,c,d$,
\begin{equation}
\Psi(a,b,c,d;r,s):= [f_r(x,y),f_s(ax+by,cx+dy)] \ge 0 
\end{equation}
A calculation shows that 
\begin{equation}
\begin{gathered}
\Psi(a,b,c,d;r,s) = a^4+b^4+c^4+d^4 + \\ 6r(a^2b^2 + c^2d^2) +
6s(a^2c^2+b^2d^2) + 6rs(a^2d^2 + 4abcd + b^2c^2).
\end{gathered}
\end{equation}
When $s = U(r)$, a  sos expression can be found: 
\begin{equation}
\begin{gathered}
2(1-r)\Psi(a,b,c,d;r,U(r)) = (1+r)(1+3r)(a^2+b^2-c^2-d^2)^2 \\
 -4 r (a^2 + c^2 - b^2 - d^2)^2 + (1 + r) (1 - 
    3 r) (a^2 + d^2 - b^2 - c^2)^2 \\ - 8 r (1 + 3 r) (a b + c d)^2, 
\end{gathered}
\end{equation}
which is non-negative when $r \in [-\frac 13, 0]$.
Note that $\Psi(1,1,1,-1;r,U(r)) = 0$; reaffirming that $[f_r,g_{U(r)}] = 0$.
\end{proof}

\begin{theorem}
Suppose $r,s \in [-\frac 13, 0]$. If $s \ge U(r)$, $p \in [[f_r]]$ and
$q \in [[f_s]]$, then $[p,q] \ge 0$. If $s < U(r)$, then there exist  $p
\in [[f_r]]$ and $q \in [[f_s]]$ so that $[p,q] < 0$.
\end{theorem}
\begin{proof}
If $0 \ge s \ge U(r)$, then $s \in [U(r),T(U(r))]$, hence
$f_s$ is a convex combination of $f_{U(r)}$ and  $f_{T(U(r))}$,
and each $f_s \circ M$ is a convex combination of $f_{U(r)}\circ
M$ and  $f_{T(U(r))}\circ M$. By Theorem 5.4,
$[f_r,f_s \circ M]$ is a convex combination of non-negative numbers
and is non-negative. 
If $U(r) \ge s \ge -\frac 13$, then $[f_r,g_s] < 0$ by \eqref{E:fglam}. 
\end{proof}

We now have the tools to analyze $B \in \mathcal B_{2,4}$.
If $Q_{2,4} \subseteq B \subseteq P_{2,4}$, let
\begin{equation}
\Delta(B) = \{ \la \in \rr : f_{\la} \in B \}.
\end{equation}  
\begin{theorem}
If $B \subset F_{2,4}$ is a blender, then $\Delta(B) =
[\tau,T(\tau)]$ for some $\tau \in [-\frac 13,0]$.
\end{theorem}
\begin{proof}
By (P2), $\Delta(B)$ is a closed interval. 
We have seen that $\Delta(P_{2,4}) = [-\frac 13,\infty)$. Since $Q_{2,4} =
P_{2,4}^* = \Sigma_{2,4}^*$, by \eqref{E:cata}, $f_{\la} \in Q_{2,4}$ if and
only if
$\left( \begin{smallmatrix}
1 & 0 & \la \\
0 & \la  & 0\\
\la & 0  & 1\\
\end{smallmatrix}\right)
$
is psd; that is,  $\Delta(Q_{2,4}) = [0,1]$.
Otherwise, let $\tau = \inf \{ \la : f_{\la} \in B \}$. Since $Q_{2,4}
\subsetneq B \subsetneq  P_{2,4}$, $\tau \in (-\frac 13, 0)$.
By (P2),
$f_{\tau} \in B$ and by (P3), $f_{T(\tau)} \in B$,
and by convexity, $f_{\nu} \in B$ for $\nu \in
[\tau,T(\tau)]$. If $\nu < \tau$, then $f_{\nu}\not \in B$ by
definition. If $\nu > T(\tau)$ and $f_{\nu} \in B$, 
then $f_{T(\nu)} \in B$ and $T(\nu) < T(T(\tau)) = \tau$, a
contradiction.
\end{proof}

Now, for $\tau \in [-\frac 13,0]$, let 
\begin{equation}
B_{\tau}:= \bigcup_{\tau \le \la \le \frac 13} [[f_\la]] = \{p: p \sim
f_{\la}, \tau \le \la \le \tfrac 13 \}  \cup \{(\al
 x +\be y)^4: \al, \be \in \rr \}.
\end{equation}

\begin{theorem}
If $B \in \mathcal B_{2,4}$, then $B =
B_{\tau}$ for some   $\tau \in [-\frac 13, 0]$ and $B_\tau^* = B_{U(\tau)}$.
\end{theorem}
\begin{proof}
Suppose $B$ is a blender and $Q_{2,4} \subsetneq B \subsetneq
P_{2,4}$. Then $\Delta(B) = [\tau,T(\tau)]$ by Theorem 5.6, so $B =
B_{\tau}$ by Prop.\!\! 5.2.  We need to show that each such $B_{\tau}$
is a blender. Since $B_0 = Q_{2,4}$ and $B_{-\frac 13} = P_{2,4}$ are blenders,
we may assume $\tau > -\frac 13$ and all $p \in B_{\tau}$ are pd. 
Clearly, (P3) holds in $B_{\tau}$.  

Suppose $p_m \in B_{\tau}$ and
$p_m \to p$. If $p$ is a 4th power, then $p \in B_{\tau}$. If $p$ is
pd, then $K(p_m) \to K(p)$ by \eqref{E:inv}, 
\eqref{E:inv2} and continuity. In any case, $K(p_m) \ge \phi(\tau)$,
so $K(p) \ge \phi(\tau)$ and $p \in B_{\tau}$. Finally, if $p \in Z_{2,4}$,
then $K(p_m) \ge \phi(\tau) > \phi(-\frac 13) = K(p)$ by Lemma 5.3, and this
contradiction completes the proof of (P2).
 
We turn to (P1). Suppose $p, q \in B_{\tau}$ and $p+q \not\in
B_{\tau}$. Since $p+q$ is pd, $p+q \sim  f_{\la}$ for some $\la < \tau$,
and so there exists $M$ so that $p\circ M + q \circ M = f_{\tau}$. But
now, \eqref{E:apo} and Theorem 5.5 give a contradiction:
\begin{equation}
0 > [f_{\la},g_{U(\tau)}] = [p\circ M,g_{U(\tau)}] + [q \circ M
,g_{U(\tau)}] \ge 0.
\end{equation} 
Thus, $p+q \in B_{\tau}$ and (P1) is satisfied, showing that
$B_{\tau}$ is a blender. It follows from Prop.\!\! 2.7 and Theorem 5.5
that $B_{\tau}^* = B_{\nu}$ for some $\nu$.  But by Theorem 5.5, 
$B_{U(\tau)} \subseteq B_{\tau}^*$ and if $\la < U(\tau)$, then
$f_{\la} \notin  B_{\nu}^*$, thus $B_{\tau}^* = B_{U(\tau)}$.
\end{proof}
A computation shows that $\phi^2(\la) + \phi^2(U(\la)) = \frac 1{27}$,
and this gives an alternate way of describing the dual
cones. Regrettably, this result was garbled in \cite[p.141]{Re1} 
 into the statement that
$B_{\tau}^* = B_{\nu}$, where $\tau^2 + \nu^2 = \frac 19$. The
self-dual blender $B_{\nu_0} =B_{\nu_0}^*$  occurs for $\nu_0 = 1 -
\sqrt{4/3}$. We know of no other interesting properties of $B_{\mu_0}$.

\section{$K_{2,2r}$: binary convex forms}

The author's Ph.D. thesis, submitted in 1976 and
published as \cite{Re0,Re00} in 1978 and 1979, 
discussed $N_{n,2r}$. (The identification of
$N_{n,2r}$ and $K_{n,2r}$ was not made there.) 
 Unbeknownst to him, V. I. Dmitriev had earlier worked on
similar questions at Kharkov University.  In
1969, S. Krein, Dmitriev's advisor, had asked about the extreme elements of
$K_{2,2r}$. Dmitriev wrote \cite{D1} in 1973 and \cite{D2} in 1991.
Dmitriev writes in \cite{D2}:  ``I am not aware of any articles on this
topic, except \cite{D1}.''  We have seen \cite{D2} both in its
original Russian and in the English translation. We have not yet  seen
\cite{D1} (although UI Interlibrary Loan is still trying!), and our
comments on \cite{D2} are based on references in \cite{D2}. There are at
least two mathematicians named V. I. Dmitriev in MathSciNet; the
author of \cite{D1,D2} is 
affiliated with  Kursk State Technical University.  

Let 
\begin{equation}\label{E:qlam}
q_{\la}(x,y) = x^6 + 6\la x^5y+ 15\la^2 x^4y^2 + 20 \la^3 x^3y^3 + 15
\la^2 x^2y^4 + 6\la x y^5 + y^6. 
\end{equation}
In the language of this paper, the four relevant results from
\cite{D1,Re00,D2} are these: 
\begin{proposition}
\ 
\smallskip

\noindent (i) $K_{2,4} = Q_{2,4}$.

\noindent (ii) $Q_{2,2r} \subsetneq K_{2,2r}$ for $r \ge 3$. 

\noindent (iii)  The elements of $\mathcal E(K_{2,6})$, are
$[[q_{\la}]]$, where $0 <   | \la | \le \frac 12$.

\noindent (iv)  $K_{3,4} \subsetneq Q_{3,4}$; specifically,
  $x^4+y^4+z^4+6x^2y^2+6x^2z^2+2y^2z^2 \in K_{3,4} \setminus
  Q_{3,4}$. 

\end{proposition}

According to \cite{D2}, \cite{D1} gave a proof of (i) and (ii) (for
even $r$); \cite{D2} gave a proof of (iii). 
All four appeared in \cite{Re00}; (iii) was announced
without proof. (The results from 
 \cite{Re00} were in the author's thesis, except
that (iv) was proved there by an extremely long perturbation argument.)
Note that (i) and (ii) follow from Prop.\!\! 3.8 and Theorems 3.10
and 3.11. Since $P_{n,m} = \Sigma_{n,m}$ if $n = 2$ or $(n,m) = (3,4)$,
 these examples are not helpful in resolving Parrilo's question about
convex forms which are not sos.  

The rest of this section discusses $\partial(K_{2,2r})$, mostly for small $r$.
Let
\begin{equation}\label{E:pdef}
p(x,y) = \sum_{i=0}^{2r} \binom{2r}i a_ix^{2r-i}y^i,
\end{equation}
and define 
\begin{equation}
\begin{gathered}
\Theta_p(x,y):=  \sum_{m=0}^{4r-4} b_mx^{4r-4-m}y^m, \quad
\text{where} \\
b_m :=  \sum_{j=0}^{2r-1}
\left(\binom{2r-2}j\binom{2r-2}{m-j}
  -\binom{2r-2}{j-1}\binom{2r-2}{m-j+1}\right)  
a_ja_{m+2-j},
\end{gathered}
\end{equation}
with the convention that $a_i=0$ if $i < 0$ or  $i > 2r$.
\begin{proposition}\cite[Prop.B]{D2}
Suppose $p \in P_{2,2r}$. Then
 $p \in K_{2,2r}$ if and only if $\Theta_p \in P_{2,4r-2}$ and
 $p \in \partial(K_{2,2r})$ if and only if $\Theta_p$ is psd but not pd.
\end{proposition} 
 \begin{proof}
 A direct computation shows that 
\begin{equation}
 \frac{\partial^2p}{\partial x^2}\frac{\partial^2p}{\partial y^2} -
 \left(\frac{\partial^2p}{\partial x\partial y}\right)^2 =
 (2r)^2(2r-1)^2\Theta_p(x,y).
\end{equation} 
Since $Hes(p;u,u) = 2r(2r-1)p(u) \ge 0$, the first assertion is proved.
Further,  $p \in \partial(K_{2,2r})$ if and only if $Hes(p;u_0,v_0) =
0$ for some $u_0 \neq 0, v_0 \neq 0$. 
\end{proof}
Observe that
$\Theta_{(\al\cdot)^{2r}} =0$, and it  may be checked that if $q(x,y)
= p(ax+by,cx+dy)$, then $\Theta_q(x,y) = 
 (ad-bc)^2\Theta_p(ax+by,cx+dy)$. Thus, if $q
 \in \partial(K_{2,2r})$, we may assume that $q \sim p$,
 where $\Theta_p(0,1) = 0$, so that 
\begin{equation}\label{E:zero}
0 = b_0 = a_0a_2 - a_1^2;\qquad 0 = b_1 = (2r-2)(a_0a_3 - a_1a_2).
\end{equation}

We give a proof that $K_{2,4} = Q_{2,4}$, using the argument of
\cite{Re00} and, presumably, \cite{D1}.
\begin{proposition}
 $K_{2,4} = Q_{2,4}$.
\end{proposition}
\begin{proof}
Suppose $q \in \mathcal E(K_{2,4})$. Then $q \in \partial(K_{2,4})$ and
$q \sim p$ where $\Theta_p$ is psd, but $\Theta_p(0,1) = 0$. If $a_0 = 0$,
then $p(0,1) = 0$, so by Prop.\!\! 4.1, $p(x,y) = a_4 y^4$ is a 4th
power. Otherwise, $a_0 > 0$, and if we write $a_1 = ra_0$, then
by \eqref{E:zero}, we have  $a_2 = r^2a_0$ and $a_3 = r^3 a_0$. Write $a_4 =
r^4a_0 + s$. A computation shows that $\Theta_p(x,y) = a_0 s
x^2(x+ry)^2$, hence $s \ge 0$ and $p(x,y) = a_0(x + ry)^4 + s
y^4$. Since $Q_{2,4} \subset K_{2,4}$ and $s \ge 0$, it follows
that $p \in \mathcal E(K_{2,4})$ if and only if $s=0$. Thus $p \in
K_{2,4}$, being  a sum of extremal elements, is a sum of 4th powers.
\end{proof}

If $2r=6$, then we shall need $\Theta_p(x,y)$ in full bloom:

\begin{equation}\label{E:theta}
\begin{gathered}
 \Theta_p(x,y) = (a_0a_2-a_1^2)x^8 +4(a_0a_3-a_1a_2)x^7y +  (6a_0a_4 +
 4a_1a_3 - 10a_2^2) x^6y^2 \\
+ 4(a_0a_5 + 4a_1a_4 - 5a_2a_3) x^5 y^3 +
(a_0a_6+14a_1a_5+5a_2a_4-20a_3^2) x^4y^4\\ + 4(a_1a_6 + 4a_2a_5 -
5a_3a_4) x^3 y^5  + (6a_2a_6 +
 4a_3a_5 - 10a_4^2) x^2y^6  \\+ 4(a_3a_6-a_4a_5)xy^7 +  (a_4a_6-a_5^2)y^8.
\end{gathered}
\end{equation}

\begin{lemma}
If $p \in K_{2,6}$ and $\Theta_p(x,y) = \ell^2(x,y)B_p(x,y)$, where
$\ell$ is linear and $B_p$ is a
pd sextic, then $p \notin \mathcal E(K_{2,6})$.
\end{lemma}
\begin{proof} 
After a linear change, we may assume $\ell(x,y) = y$, and assume $p$
is given by \eqref{E:pdef}, so that \eqref{E:theta} holds. If
$a_0=p(1,0) = 0$, then as 
in Prop.\!\! 6.3, $p(x,y) = a_6y^6$ and $\Theta_p(x,y) = 0$. Otherwise,
we again have  $a_1 = ra_0$, $a_2 = r^2a_0$ and $a_3 = r^3 a_0$. A
computation shows that 
\begin{equation}\label{E:bp}
\begin{gathered}
B_p(x,y) = 6a_0(a_4-r^4a_0)x^6 + 4a_0(a_5+4ra_4-5r^5a_0)x^5y  \\ +
a_0(a_6+14ra_5+5r^2a_4-20r^6a_0) x^4y^2\\ + 4ra_0(a_6 + 4ra_5 -
5r^2a_4) x^3 y^3+  (6r^2a_0a_4 +
 4r^3a_0a_5 - 10a_4^2) x^2y^4 \\ + 4(r^3a_0a_6-a_4a_5)xy^5 +  (a_4a_6-a_5^2)y^6.
\end{gathered}
\end{equation}
Observe that if $p_{\la} = p + \la y^6$, then $a_6$ is replaced above
by $a_6 + \la$ and
\begin{equation}
\begin{gathered}
B_{p_{\la}} = B_p+ \la (a_0 x^4y^2 + 4r a_0 x^3y^3 + 6r^2
a_0 x^2y^4 + 4r^3 a_0 xy^5 + a_4 y^6). 
\end{gathered}
\end{equation}
Since $B_p$ is pd, there exists sufficiently small $\epsilon$ so that
$B_{p_{_{\pm \epsilon}}}$ is psd, so $p_{\pm \epsilon} \in K_{2,6}$. 
But then $p = \frac 12(p_{\epsilon} + p_{-\epsilon})$ is not extremal.
\end{proof}

\begin{proof}[Proof of Prop.\!\! 6.1(iii)]
By Prop.\! 6.2 and Lemma 6.4, we may assume that $\Theta_p = y^2B_p$ and $B_p$
is psd, but not pd. If $B_p(0,1) = 0$, then by \eqref{E:bp}, $a_4 = r^4a_0$
and $a_5 = r^5a_0$ and, as before, if $a_6 = r^6a_0 + t$, then $\Theta_p =
a t x^4(x+ry)^4$, so $t \ge 0$ and $p \in \mathcal E(K_{2,6})$ if and
only if $t=0$, so $p$ is a 6th power. 

If $B_p(1,e) = 0$ and $e \neq 0$, and $\tilde p(x,y) = p(y,x+ey)$, then
$\Theta_{\tilde p}(x,y) = 0$ at $(x,y) = (1,0), (0,1)$, and by dropping the
tilde, we may assume from \eqref{E:theta} that $0 = a_4a_6 - a_5^2 = a_3a_6
- a_4a_5$. Again, $a_6 = p(0,1) \ge 0$, and if $a_6=0$, then $p$ is a
6th power. Otherwise, we set $a_5 = s a_6$, so that $a_4 = s^2a_6$ and
$a_3 = s^3 a_6$; recall that $a_3 = r^3 a_0$ as well. If $s=0$, then
$a_3 = 0$, so $r=0$ and $p(x,y) = a_0x^6 + a_6y^6$, which is only
extremal if it is a 6th power. Thus $s \neq 0$, and similarly, $r \neq 0$. 
Letting $t = s^{-1}$, we obtain the formulation of \cite{D2}:
\begin{equation}
p(x,y) = a_0(x^6 + 6r x^5y + 15 r^2 x^4y^2 + 20 r^3 x^3y^3 + 15 r^3t
x^2y^4 + 6 r^3t^2 xy^5 + r^3t^3 y^6)
\end{equation} 
Finally, send $(x,y) \mapsto (a_0^{-1/6}x,
a_0^{-1/6}(rt)^{-1/2} y)$ and set $\la = \sqrt{r/t} = \sqrt{rs}$ to obtain
  $q_{\la}$.

 A calculation shows that
\begin{equation}
\begin{gathered}
\Theta_{q_{\la}}(x,y) = (1-\la^2)x^2y^2 C_{\la}(x,y),\quad \text{where} \\
C_{\la}(x,y) =  6\la^2(x^4+y^4) + (4\la +
20\la^3)(x^3y+xy^3) + (1+15\la^2+20\la^4)x^2y^2.
\end{gathered}
\end{equation} 
Note that 
\begin{equation}
\begin{gathered}
D_{\la}(x,y):= C_{\la}(x+y,x-y)= (1 + \la) (1 + 2 \la) (1 + 5 \la + 10
\la^2)x^4\\ -2 
(1-\la^2)(1-20 \la^2)x^2y^2 + (1 -\la) (1 - 2 \la) (1 - 5 \la + 10 \la^2)x^4.
\end{gathered}
\end{equation} 
If $\Theta_{q_{\la}}$ is psd, then   $6\la^2(1-\la^2) \ge 0$, so $|\la| \le
1$. Under this assumption, it suffices to determine when $D_{\la}$ is
psd. Since $D_{\la}(1,0), D_{\la}(0,1) \ge 0$,  $|\la| \le \frac 12$. 
If $D_{\la}(x,y) = E_{\la}(x^2,y^2)$, then 
the discriminant of $E_{\la}$ is
$128\la^2(1-\la^2)(1-10\la^2)$, hence $D_{\la}$ is psd if $0 \le \la^2
\le \frac 1{10}$. But, if $\frac 1{20} \le \la^2 \le \frac 14$, then
$D_{\la}$ is a sum of psd monomials. Thus $D_{\la}$ is psd
if $|\la| \le \frac 12$, and hence this is also true for 
 $C_{\la}$ and  thus for $\Theta_{q_{\la}}$, so $q_{\la} \in K_{2,6}$. 
\end{proof}

Since  $\Theta_{q_\la}$ has two zeros when $|\la| < \frac 12$, but
$\Theta_{q_{1/2}} = \frac 98 x^2y^2(x+y)^2(x^2+xy+y^2)$ has
three, one expects that  the algebraic patterns for $\Theta_p$ will
be variable for  $ p \in \mathcal E(K_{2,2r})$ for $r \ge 3$ and that 
 $\mathcal E(K_{2,2r})$ will be hard to analyze.

Note also that
\begin{equation}\label{E:even}
\begin{gathered}
q_{\la}(x+y,x-y) = 
2(1+\la)(1+5\la+10\la^2)x^6 + 30(1-\la^2)(1+2\la)x^4y^2 \\ +
30(1-\la^2)(1-2\la)x^2y^4 
+  2(1-\la)(1-5\la+10\la^2)y^6.
\end{gathered}
\end{equation}
One of the two boundary examples is $q_{-1/2}(x+y,x-y)= x^6 + 45 x^2 y^4 + 18
y^6$, which scales to $x^6 + 15\al x^2 y^4 + y^6$, where $\al^3 =
\frac 1{12}$. 

We now consider the sections of $P_{2,6}=\Sigma_{2,6}$,
$Q_{2,6}$ and $K_{2,6}$ consisting of forms 
\begin{equation}
g_{A,B}(x,y) = x^6 + \binom 62 A x^4y^2 + \binom 64 B x^2y^4 + y^6,
\end{equation}
and identify $g_{A,B}$ with the point $(A,B)$ in the plane.

If $g_{A,B}$ is on the boundary of the $P_{2,6}$ section, then it
is not pd, and we may assume $(x + r y)^2\ |\ g_{A,B}$ for some $r \neq
0$.  Thus, $(x-ry)^2\ |\ g_{A,B}$ as well, and 
since the remaining factor must be even, the coefficients of $x^6,y^6$ force it
to be $x^2 + \frac 1{r^4} y^2$. Thus, the boundary forms for the
section of $P_{2,6}$ are
\begin{equation}
(x^2-r^2y^2)^2(x^2 + \tfrac 1{r^4}y^2) = x^6 +( \tfrac 1{r^4} -
2r^2)x^4y^2 +  (r^4 - \tfrac 2{r^2})x^2y^4 + y^6.
\end{equation}
The parameterized boundary curve
\begin{equation}
(A,B) = \tfrac 1{15}( \tfrac 1{r^4} - 2r^2, r^4 - \tfrac 2{r^2})
\end{equation}
is strictly decreasing as we move from left to right, and is a
component of the curve $500(A^3+B^3) = 1875(AB)^2 + 150AB - 1$.

By \eqref{E:cata}, $g_{A,B}$ is in $Q_{2,6} = \Sigma^*_{2,6}$, iff
$\left( 
\begin{smallmatrix} 1 & 0 & A & 0 \\ 0 & A & 0 & B \\ A & 0 & B & 0 \\ 0 &
  B & 0 & 1  
\end{smallmatrix}\right)$
is psd iff $A \ge B^2$ and $B \ge A^2$, so the section is the
familiar region between these two parabolas. 

Except for the fortuitous identity \eqref{E:even},
 it would have been very challenging to determine the section for $K_{2,6}$. 
Scale $x$ and $y$ in \eqref{E:even} to get $g_{A,B}$: the
parameterization of the boundary is 
$(\psi(\la),\psi(-\la))$, where 
\begin{equation}
\psi(\la) = \frac
{(1-\la)^{2/3}(1+\la)^{1/3}(1+2\la)}{(1+5\la+10\la^2)^{2/3}(1-5\la+10\la^2)^{1/3}}. 
\end{equation}
The intercepts occur when $\la = \pm \frac 12$ and are 
$(12^{-\frac 13},0)$ and $(0,12^{-\frac 13})$. The point $(1,1)\  (\la = 0)$  is
smooth but of infinite curvature. The Taylor series of $\psi(\la)$ at
$\la=0$ begins $1 + \frac {16}3 \la^3 - 48 \la^4$, so locally, $x-y
\approx \frac 
{32}3 \la^3$ and $x+y-2 \approx -96 \la^4$, hence  
\begin{equation*}
x+y-2 \approx
-\tfrac{3^{7/3}}{2^{5/3}}(x-y)^{4/3}. 
\end{equation*}
The maximum value of $\psi(\la)$ is
$5^{-5/3}(1565+496\sqrt{10})^{1/3} \approx 1.000905$ at $\la =
\frac{2\sqrt{10}-5}{15} \approx .0883$; this was asserted without
proof in \cite[p.232]{Re00}. 

 At this point, we punt
and present some trinomials in $\partial(K_{2,2r})$. Suppose
 $1 \le v \le 2r-1$, $a,c > 0$ and suppose
\begin{equation}
h(x,y) = a x^{2r} + b x^{2r-v}y^v + c y^{2r} \in K_{2,2r}.
\end{equation}
An examination of the end terms of $\Theta_h$ shows that $v$ must be
even and $b \ge 0$. If $b=0$, then $h \in Q_{2,2r}$, so we assume $b >
0$, and wish to find the largest possible value of $b$. Calculations,
which we omit, show that if 
 \begin{equation}\label{E:hrk}
\begin{gathered}
h_{r,k}(x,y) := (r-k)(2(r-k)-1)^2 x^{2r}\\
 + r(2r-1)(2k-1)(2r-2k-1)x^{2r-2k}y^{2k} +
k(2k-1)^2 y^{2r}, 
\end{gathered}
\end{equation}
then $\Theta_{h_{r,k}}(x,y)= x^{2r-2-2k}y^{2k-2}(x^2-y^2)^2g(x,y)$,
where $g$ is a (psd) sum 
of even terms with positive coefficients, and that
if $c > 0$ and $g_{r,k,c} = h_{r,k} + c x^{2r-2k}y^{2k}$, 
 then $\Theta_{g_{r,k,c}}(1,1) < 0$. Given $(a,c)$, there exist
   $(\al,\be)$ so that the coefficients of $x^{2r}$ and 
$y^{2r}$ in $h_{r,k}(\al x, \be y)$ are both 1, and we get the examples in
\cite[Prop.1]{Re00}. In particular,  
\begin{equation}
h_{4k,2k}(x,y)  \sim x^{4k} + (8k-2)x^{2k}y^{2k} + y^{4k}\in \partial(K_{2,4k}).
\end{equation}
Similar methods show that  
\begin{equation}
x^{6k} + (6k-1)(6k-3)x^{4k}y^{2k} + (6k-1)(6k-3)x^{2k}y^{4k} + y^{6k}
\in \partial(K_{2,6k}).
\end{equation}

We have been unable to analyze $K_{2,8}$ completely, but have found
this  interesting element in $\mathcal E(K_{2,8})$:
\begin{equation}
p(x,y) = (x^2+y^2)^4 + \tfrac 8{\sqrt 7}\ x y (x^2 - y^2)(x^2+y^2)^2,
\end{equation}
for which $\Theta_p(x,y)= 3072 x^2 (x - y)^2 y^2 (x + y)^2 (x^2 + y^2)^2$.

\section{Sums of 4th powers and octics}

Hilbert's 17th Problem asks whether $p \in P_{n,2r}$ must be a sum of
squares of rational functions: does there always exist $h = h_p \in
F_{n,d}$ (for some $d$) so that $h^2p \in \Sigma_{n,2r+2d} = W_{n,2(r+d)}$? Artin
proved that the 
answer is ``yes''. (See \cite{R2,Re3}.) Becker \cite{Be} investigated the
question for higher even powers. His result implies that if $p \in
P_{2,2kr}$ and all real linear factors of $p$ (if any) occur to an exponent
which is a multiple of $2k$, then there exists $h = h_p \in F_{2,d}$
(for some $d$) so that $h^{2k}p \in W_{2,(r+d,2k)}$.

For example, by Becker's criteria, $f_{\lambda}$ (c.f. \eqref{E:flam})
is a sum of 
4th powers of rational functions if and only if it is pd; that is,
$\la \in (-\frac 13,\infty)$. As we have seen, $f_{\lambda} 
\in Q_{2,4} = W_{2,(1,4)}$ if and only if $\la \in [0,1]$. If $\ell$ is linear and
$\ell^4f = \sum_k h_k^4 \in W_{2,(2,4)}$, then $\ell | h_k$, so if $f_{\la}
\notin Q_{2,4}$ and $h^4f \in W_{2,(1+d,4)}$, then  $\deg h = d \ge
2$. The  identity 
\begin{equation}
\begin{gathered}
3(3x^4 - 4x^2y^2 + 3y^4)(x^2 + y^2)^4 \\ = 2 ( (x-y)^4 + (x+y)^4)
(x^8 + y^8) + 5x^{12} + 11x^8y^4 + 11x^4y^8 + 5y^{12} 
\end{gathered}
\end{equation}
shows that $(x^2+y^2)^4f_{\la} \in W_{2,(3,4)}$ for $\la \in [-\frac
29, \frac{11}3]$, since $T(-\frac{2}9) = \frac{11}3$, c.f. \eqref{E:TU}.  

We know no alternate characterization of 
$W_{2,(u,4)}$, but offer the following  conjecture:
\begin{conjecture}
If $p \in P_{2,4u}$, then $p \in W_{2,(u,4)}$ if and only if there
exist $f,g \in P_{2,2u}$ so that $p = f^2 + g^2$.
\end{conjecture}

 It follows from \eqref{E:h22} that
 the square of a psd binary form is a sum of three 4th
powers.  Conjecture 7.1 thus implies that any sum of 4th powers of
polynomials is a sum of six 4th powers of polynomials.
Any sum of $s$ 4th powers will be a sum of $s$
squares of psd forms; the conjecture asserts that $p$ is a
sum of {\it two} such squares.
If $p \in W_{2,(u,4)}$, then $p \in P_{2,4u} =
\Sigma_{2,4u}$, so $p = f^2 + g^2$ for some $f,g \in F_{n,2u}$; 
the conjecture says that there is a representation
in which $f$ and $g$ are themselves psd. 

This seems related to a  result in \cite{CLPR} about sums of 4th powers
of rational functions over real closed fields. If $p = \sum h_k^4$ and
$\ell | p$ for a linear form, then $\ell^{4t} | p$ for some $t$ 
and $\ell^t  | h_k$, so we may assume $p$ is pd. The following is a
special case of \cite[Thm.4.12]{CLPR}, referring to
sums of 4th powers of non-homogeneous rational functions.
\begin{proposition}
Suppose $p \in \rr[x]$ is pd. Then $p$ is a sum of 4th powers in
$\rr(x)$ if and only if there exist pd $f,g,h$ in $\rr[x]$,
 $\deg f = \deg g$, such that $h^2p = f^2 +g^2$.
\end{proposition}
It follows that a sum of 4th powers in $\rr(x)$ is a sum of at most six
4th powers.

\begin{theorem}
Conjecture 7.1 is true for $p \in W_{2,(1,4)} = Q_{2,4}$.
\end{theorem}
\begin{proof}
We have seen that if $p \in  W_{2,(1,4)}$, then $p \sim f_{\la}$ for
$\la \in [0,1]$. If $\la \in (\frac 13, 1]$, then $T(\la) \in [0,\frac
13)$, so it suffices to find a representation for $F_{\la}$ with $\la
\in [0,\frac 13]$. Such a representation is  
$f_{\la}(x,y) = (x^2 + 3\la y^2)^2 + (1-9\la^2)(y^2)^2$.
\end{proof}

\begin{theorem}
Conjecture 7.1 is true for even symmetric octics.
\end{theorem}

It will take some work to get to the proof of Theorem 7.4. 
For the rest of this section, write $W:= W_{2,(2,4)}$. We first characterize
$\partial(W^*)$. 

\begin{theorem}
If $p \in \partial(W^*)$, then $p = (\al\cdot)^8$ or $p \sim q$, where
\begin{equation}\label{E:bdry}
q(x,y) = d_0 x^8 + 8d_1 x^7y + 28d_2 x^6y^2 + 28 d_6 x^2y^6 + 8 d_7 x
y^7 + d_8 y^8, 
\end{equation}
and
\begin{equation}\label{E:disc}
(6d_2u^2 + 6d_6 w^2)(d_0 u^4
+ 4d_2u^3w + 4d_6uw^3  + d_8 w^4)-( 2d_1u^3+2d_7w^3)^2
\end{equation}
is psd.
\end{theorem}
\begin{proof}
Consider a typical element $q \in W^*$,
\begin{equation}\label{E:octdef}
q(x,y) = \sum_{k=0}^8 \binom 8k d_k x^{8-k}y^k.
\end{equation}
Then as in Prop.\!\! 3.9, 
\begin{equation}\label{E:Hoct}
\begin{gathered}[]
H_q(u,v,w):= [q,(u x^2 + v x y + w y^2)^4] = d_0u^4  + 4 d_1 u^3 v + d_2(6 u^2 v^2
 + 4 u^3 w)  \\ +
d_3( 4 u v^3  + 12 u^2 v w) + d_4(v^4  + 12 u v^2 w  + 
 6 u^2 w^2) + d_5(4 v^3 w  + 12 u v w^2) \\ + d_6(6 v^2 w^2  + 
 4 u w^3) + 4 d_7v w^3  + d_8 w^4 
\end{gathered}
\end{equation}
is a psd ternary quartic in $u,v,w$.
If $ q \in \partial(W^*)$, then $[q,h^2] = 0$ for some non-zero
quadratic $h$. Since $\pm h \sim x^2, xy, x^2+y^2$, 
it suffices by Prop.\!\! 2.6 to consider three cases: $[q,x^8]=0,
[q,x^4y^4]=0$ and $[q,(x^2+y^2)^4] = 0$. Since  
\begin{equation}\label{E:h24}
420(x^2+y^2)^4 = 256(x^8+y^8) + \sum_{\pm} (x \pm \sqrt 3 y)^8 +
( \sqrt 3 x \pm y)^8, 
\end{equation}
 $[q,(x^2+y^2)^4] = 0$ implies that $q(1,0) = q(0,1) = q(1, \pm \sqrt
 3) = q(\sqrt3 , \pm 1) = 0$; since $q$ is psd,  $q=0$. (An alternate
 proof derives 
 this result from $(x^2 + y^2)^4 \in int(Q_{2,8})$ by 
 \cite[Thm.8,15(ii)]{Re1}, so $(x^2+y^2)^4 \in int(W)$.)

Suppose $[h,(x^2)^4]=0$; that is,
$H_q(1,0,0) = 0$. Then 
$d_0 =0$, and since $H_q$ is now at most quadratic in $u$, it follows
that $d_1=d_2 = 0$. This implies that the coefficient of $u^2$ in
$H_q$ is $12d_3 vw + 6d_4w^2$, hence $d_3=0$ and
\begin{equation}
\begin{gathered}
H_q(u,v,w) = u^2(6d_4w^2) + 2u(2d_6w^3 + 6d_5vw^2+6d_4v^2w) \\ + (d_8w^4
+ 4d_7 w^3v+6d_6w^2v^2+4d_5 wv^3+ d_4v^4).
\end{gathered}
\end{equation}
Since $H_q$ is psd if and only if its discriminant with respect to $u$ is
psd in $v,w$, and this discriminant is $-30d_4^2 v^4w^2 + $ lower terms in $v$,
$d_4=0$. Since $H_q$ cannot be linear in $u$, it follows that
$d_5=d_6=0$ and $H_q(u,v,w) = d_8w^4 + 4d_7w^3v$, which is only psd
if $d_7=0$, so that $q(x,y) = d_8y^8$ is an 8th power. 

Finally, suppose $[q,x^4y^4] = 0$; that is, $H_q(0,1,0) = d_4 =
0$. Since $H_q$ is at most 
quadratic in $v$, it follows that $d_3=d_5 = 0$ as well, so $q$ has the
shape \eqref{E:bdry} and
\begin{equation}
\begin{gathered}
H_q(u,v,w) =  v^2(6d_2u^2 + 6d_6 w^2)  \\ + 2v( 2d_1u^3+2d_7w^3) + d_0 u^4
+ 4u^3w d_2 + 4uw^3 d_6 + d_8 w^4;
\end{gathered}
\end{equation}
$H_q$ is psd if and only if its discriminant with respect to $v$,
namely \eqref{E:disc}, is psd.
\end{proof}

It should be possible to characterize $\mathcal E(W^*)$, though we do
not do so here. One family of extremal elements is parameterized by
 $\al \in \rr$: 
\begin{equation}
\omega_{\al}(x,y):= x^8 + 28 x^2 y^6 + 24  \al x y^7 + 3(1 + 2\al^2) y^8
\in \mathcal E(W^*). 
\end{equation}
In this case, 
\begin{equation}
\begin{gathered}
H_{\om_{\al}}(u,v,w) = 6 v^2 w^2 + 12 \al v w^3 +  u^4  + 4 u w^3 + (3 +
6\al^2)w^4 \\ = 6 (v w + \al w^2)^2 + (u+w)^2(u^2-2uw+3w^2) 
\end{gathered}
\end{equation}
is psd; $H_{\om_\al}(0,1,0) = H_{\om_\al}(1,\al,-1) = 0$,
and $H_{\om_\al}(u,v,0)=u^4$ has a 4th order zero at $(0,1,0)$.
It is unclear whether $\om_{\al}$ has other interesting algebraic properties.

We now simplify matters by limiting our attention to even symmetric
octics. Let 
\begin{equation}\label{E:tildeF}
\widetilde F = \{ ((A,B,C)):= A x^8 + B x^6y^2 + C x^4y^4 + B x^2y^6 +
A y^8\ : \ A, 
B, C \in \rr\}.
\end{equation}
denote the cone of even symmetric octics, and let
\begin{equation}
\widetilde W = W \cap \widetilde F.
\end{equation}
Then $\widetilde W$ is no longer a blender, because (P3) fails
spectacularly. However, it is still a closed convex cone. We give
the inner product explicitly:
\begin{equation}\label{E:esoip}
p_i = ((A_i,B_i,C_i)) \implies [p_1,p_2] = A_1A_2 + \tfrac{B_1B_2}{28}
+ \tfrac{C_1C_2}{70} +  \tfrac{B_1B_2}{28} +  A_1A_2.
\end{equation}
 Let $(\widetilde W)^* \subset \widetilde F$
denote the dual cone to $\widetilde W$. Here is a special
case of \cite[p.142]{Re1}.
\begin{theorem}
$(\widetilde W)^* = W^* \cap \widetilde F$.
\end{theorem}
\begin{proof}
 Suppose $p \in \widetilde W$
and $q \in  W^* \cap \widetilde F$. Then  $p \in W$ and $q \in W^*$ imply
$[p,q] \ge 0$, so $q \in (\widetilde W)^*$. Suppose now that $q \in
(\widetilde W)^*$; we wish to show that $q \in W^*$.  Pick $r \in W$,
and let  $r_1 = r$, $r_2(x,y) = r(x,-y)$, $r_3(x,y) = r(y,x)$ and
$r_4(x,y) = r(y,-x)$. Since $q \in \widetilde F$, $[r_j,q] = [r,q]$
for $1 \le j \le 4$, and since  $p = r_1+r_2+r_3+r_4 \in \widetilde W$,
$0 \le [p,q] = 4[r,q]$. Thus, $[r,q] \ge 0$ as desired.
 \end{proof}
We need not completely analyze $(\widetilde W)^*$ to determine
$\widetilde W$. The following suffices. 
\begin{lemma}
If $q =((1,0,0))$,
$((4,28,0))$ or $((6-4\la^2+3\la^4, 28(6-\la^2), 420))$, $\la \in
\rr$,  then $q \in W^*$.
\end{lemma}
\begin{proof}
Using the notation of \eqref{E:octdef}, suppose 
\begin{equation}
q(x,y) = ((d_0,28d_2,70d_4)) = d_0^8 + 28 d_2x^6y^2 + 70 d_4 x^4y^4 +
28 d_2 x^2y^6 + d_0 y^8. 
\end{equation}
Comparison with \eqref{E:esoip}  shows that
\begin{equation}\label{E:wstar}
q \in \widetilde W^* \iff ((A,B,C)) \in \widetilde W \implies 2d_0 A +
2d_2 B +d_4 C \ge 0. 
\end{equation}
On the other hand, \eqref{E:Hoct} and Theorem 7.6 imply that $q \in
\widetilde W^*$ if and only if
\begin{equation}
H_q(u,v,w) = d_0(u^4+w^4) + d_2(u^2+w^2)(6v^2 + 4uw) + d_4(v^4 +
12uv^2w +6u^2w^2)  
\end{equation}
is psd. If $(d_0,d_2,d_4) = (1,0,0)$, then $H_q(u,v,w) =
u^4+w^4$, which is psd,
and if $(d_0,d_2,d_4) = (4,1,0)$, then 
\begin{equation}
H_q(u,v,w) =  4(u+w)^2(u^2-uw+w^2) + 6(u^2+w^2)v^2.
\end{equation}
Finally, if $(d_0,d_2,d_4) = (6-4\la^2+3\la^4, 6-\la^2, 6)$, then a
computation gives
\begin{equation}
\begin{gathered}
2H_q(u,v,w) = 2(6-4\la^2+3\la^4)(u^4+w^4)  \\ +
2(6-\la^2)(u^2+w^2)(6v^2 + 4uw) + 
12(v^4 + 12uv^2w +6u^2w^2)  \\
= 48(u+w)^2v^2 + 4\la^2(u+w)^4 + 3\la^4(u^2-w^2)^2 \\ +
3(2v^2 + 2(u+w)^2 - \la^2(u^2+w^2))^2.
\end{gathered}
\end{equation}
Note that $H_q(1,\pm \la, -1) = 0$.
\end{proof}

An important family of elements in $\widetilde W$ is
\begin{equation}
\begin{gathered}
\psi_\la(x,y) : = \tfrac 12\left( (x^2 + \la xy - y^2)^4 + (x^2 - \la xy -
  y^2)^4 \right)\\ = 
((1,\ 6\la^2-4,\ \la^4-12\la^2+6))
\end{gathered}
\end{equation}
\begin{theorem}
The extremal elements of $\widetilde W$ are $x^4y^4$ and $\{\psi_{\la} :
 \la \ge 0\}$. Hence $p =((A,B,C)) \in \widetilde W$ if and only if
\begin{equation}\label{E:cond}
\begin{gathered} 
A = B=0,\ C \ge 0,\quad \text{or}\quad  A > 0,\ B \ge - 4A,\ 36AC \ge
B^2 - 64AB - 56A^2.  
\end{gathered}
\end{equation}
\end{theorem}
\begin{proof}
By Lemma 7.7 and \eqref{E:wstar}, if $p \in \widetilde W$, then $A\ge
0$, $A + 4B \ge 0$ and 
\begin{equation}\label{E:hcond}
2(6 - 4\la^2 + 3\la^4)A + 2(6 - \la^2) B + 6C \ge 0. 
\end{equation}
We have $A = p(1,0)=p(0,1)\ge 0$, and if $A=0$ and $p = \sum h_k^4$, then $xy
| h_k$, hence $p = [0,0,C]$ with $C \ge 0$. Otherwise, assume that $A =
1$, so that \eqref{E:cond} becomes
\begin{equation}
B \ge -4, \quad C \ge \tfrac 1{36}(B^2 - 64B -56).
\end{equation}
The first inequality follows from $((4,28,0)) \in \widetilde W^*$, and
we can thus write $B = 
6\al^2 - 4$, where $\al = \sqrt{\frac {B+4}6}$. Put $\la = \al$ in
\eqref{E:hcond} to obtain 
\begin{equation}\label{E:parab}
C \ge \al^4 - 12\al^2 + 6 = \tfrac 1{36}(B^2 - 64B -56).
\end{equation}

Conversely, suppose $p=((A,B,C))$ satisfies \eqref{E:cond}. If $A = 0$, then
$p = c x^4y^4 \in \widetilde W$. If $A > 0$, then we can take $A=1$ and
substitute $B = 6\al^2 - 4$, so that, by \eqref{E:parab},
\begin{equation}
p = ((1,B,C)) = ((1,6\al^2-4, \al^4-12\al^2-6)) + ((0,0,\gamma)) =
\psi_{\la}(x,y) + \gamma x^4y^4
\end{equation}
for some $\gamma \ge 0$, hence $p \in \widetilde W$.
\end{proof}

Taking $(A,B) = (1,0)$, we obtain \eqref{E:48}. Suppose $\la, \mu \ge
-2 $. Then Theorem 7.6 implies that (c.f. \eqref{E:flam}) 
$f_{\la}(x,y)f_{\mu}(x,y) \in W$ if and only if
\begin{equation}
(17 -12\sqrt 2) (\la+2) \le \mu+2 \le (17 + 12\sqrt 2) (\la+2)
\end{equation}
There is a peculiar resonance with the example after Theorem 4.7.

\begin{proof}[Proof of Theorem 7.4]
Suppose the even symmetric octic 
 $((A,B,C))$ satisfies \eqref{E:cond}. If $A=0$, then $((0,0,C)) =
C(x^2y^2)^2$. Otherwise, again suppose $A=1$ and write $B = 6\al^2-4$, so
\begin{equation}
B = 6\al^2-4, \quad C =  \tfrac 1{36}(B^2 - 64B -56) + T =
 \al^4 - 12\al^2 + 6 + T, \quad T \ge 0. 
\end{equation}
 Observe that
\begin{equation}
\begin{gathered}
(x^4 + (3\al^2-2) x^2y^2 + y^4)^2 + (T- 8\al^4)(x^2y^2)^2 \\=
((1,6\al^2-4,9\al^4 - 12 \al^2 + 6)) + ((0,0,T-8\al^4)) = ((1,B,C)),
\end{gathered}
\end{equation}
so if $T \ge 8\al^4$, then we are done. Otherwise, $0 \le T \le 8\al^4$. 
Finally, note that
\begin{equation}
\begin{gathered}
\tfrac12 \left( \bigl((x^2 - \sqrt{\la} x y - y^2)^2 + \mu x^2 y^2\bigr)^2 +
  \bigl((x^2 + \sqrt{\la}  x y - y^2)^2 + \mu x^2 y^2\bigr)^2\right)\\ = ((1,
6\la+2\mu-4, 
6-12\la+\la^2-4\mu+2\la\mu+\mu^2)) 
\end{gathered}
\end{equation}
is a sum of two squares of psd forms if $\mu \ge 0$. One
solution to the system
\begin{equation}
\begin{gathered}
 6\al^2-4=  6\la+2\mu-4,  \al^4 - 12\al^2 + 6 +  T =
 6-12\la+\la^2-4\mu+2\la\mu+\mu^2  
\end{gathered}
\end{equation}
is  
\begin{equation}
\begin{gathered}
\la = \frac{ 3\al^2 - \sqrt{\al^4+T}}2, \quad 
 \mu = \frac{3(\sqrt{\al^4+T} - \al^2)}2.
\end{gathered}
\end{equation}
Evidently, $\mu \ge 0$; since $T \le 8\al^4$, $\la \ge 0$, so
$\sqrt{\la}$ is real. 
\end{proof}

\section{Bibliography}

\end{document}